\documentclass[11pt,letterpaper]{amsart}

\usepackage[margin=1.08in]{geometry}
\usepackage{amsmath,amssymb,amsthm,mathtools}
\usepackage{microtype}
\usepackage{xcolor}
\usepackage[T1]{fontenc}
\usepackage{libertinus}
\usepackage{enumitem}
\usepackage{booktabs,tabularx}
\usepackage{array}
\usepackage{hyperref}
\usepackage{aliascnt}
\usepackage[nameinlink,capitalize]{cleveref}

\hypersetup{
  colorlinks=true,
  linkcolor=black,
  citecolor=black,
  urlcolor=black,
  pdftitle={Hadamard rigidity and sharp stability in the completely bounded Bohnenblust--Hille inequality},
  pdfauthor={Daniel Nunez-Alarcon, Daniel M. Pellegrino and Eduardo V. Teixeira}
}
\setlist{itemsep=0.25em,topsep=0.45em}
\allowdisplaybreaks
\numberwithin{equation}{section}
\newtheorem{maintheorem}{Theorem}

\newaliascnt{proposition}{theorem}
\newtheorem{proposition}[proposition]{Proposition}
\aliascntresetthe{proposition}
\newaliascnt{lemma}{theorem}
\newtheorem{lemma}[lemma]{Lemma}
\aliascntresetthe{lemma}
\newaliascnt{corollary}{theorem}
\newtheorem{corollary}[corollary]{Corollary}
\aliascntresetthe{corollary}
\newaliascnt{conjecture}{theorem}

\aliascntresetthe{conjecture}
\theoremstyle{definition}
\newaliascnt{definition}{theorem}
\newtheorem{definition}[definition]{Definition}
\aliascntresetthe{definition}
\theoremstyle{remark}
\newaliascnt{remark}{theorem}
\newtheorem{remark}[remark]{Remark}
\aliascntresetthe{remark}

\crefname{maintheorem}{theorem}{theorems}
\Crefname{maintheorem}{Theorem}{Theorems}
\crefname{theorem}{theorem}{theorems}
\Crefname{theorem}{Theorem}{Theorems}
\crefname{proposition}{proposition}{propositions}
\Crefname{proposition}{Proposition}{Propositions}
\crefname{lemma}{lemma}{lemmas}
\Crefname{lemma}{Lemma}{Lemmas}
\crefname{corollary}{corollary}{corollaries}
\Crefname{corollary}{Corollary}{Corollaries}
\crefname{conjecture}{conjecture}{conjectures}
\Crefname{conjecture}{Conjecture}{Conjectures}
\crefname{definition}{definition}{definitions}
\Crefname{definition}{Definition}{Definitions}
\crefname{remark}{remark}{remarks}
\Crefname{remark}{Remark}{Remarks}

\DeclareMathOperator{\Tr}{Tr}

\title[Hadamard rigidity and sharp stability]{Hadamard Rigidity and Sharp Stability in the Completely Bounded Bohnenblust--Hille Inequality}

\author[D. N\'u\~nez-Alarc\'on]{Daniel N\'u\~nez-Alarc\'on}
\address{Department of Mathematics, Universidad Nacional de Colombia, Bogot\'a, Colombia}
\email{dnuneza@unal.edu.co}

\author[D. M. Pellegrino]{Daniel M. Pellegrino}
\address{Departamento de Matem\'atica, Universidade Federal da Para\'iba, Jo\~ao Pessoa, PB, Brazil}
\email{daniel.pellegrino@academico.ufpb.br}

\author[E. V. Teixeira]{Eduardo V. Teixeira}
\address{Department of Mathematics, Oklahoma State University, Stillwater, OK 74078, USA}
\email{eduardo.teixeira@okstate.edu}


\begin{document}

\begin{abstract}
The completely bounded Bohnenblust--Hille inequality of Arunachalam, Dutt, Escudero Guti\'errez and Palazuelos controls coefficient summability with optimal constant one, independently of the ambient dimension. We classify all nonzero equality cases over both the real and complex fields: they are precisely scalar multiples of Hadamard chains supported on equal-sided Cartesian boxes. Thus equality determines both the support and the factorization of the coefficient tensor. For each fixed degree $d\ge2$, we also establish sharp structural stability. A form with completely bounded norm one and deficit $\varepsilon$ lies within $C_d\sqrt\varepsilon$, in the critical coefficient $\ell_{2d/(d+1)}$ norm, of both a flat unimodular path and a normalized unitary chain on the same box. The corresponding unimodular and scaled unitary edges satisfy a matching normalized Frobenius estimate. At zero deficit both approximants coincide with the original Hadamard chain. All constants are independent of the ambient dimension, and the exponent $1/2$ is optimal for both approximations over either field.
\end{abstract}
\maketitle

\section{Introduction}

The Bohnenblust--Hille inequality originates in the 1931 work of Bohnenblust and Hille on the absolute convergence of Dirichlet series \cite{BH}. In its multilinear form, for $\mathbb K\in\{\mathbb R,\mathbb C\}$, it asserts that for every integer $d\ge2$ there is a constant $B_d$, independent of the dimension $n$, such that
\[
 \left(
 \sum_{i_1,\ldots,i_d=1}^n
 |T(e_{i_1},\ldots,e_{i_d})|^{\frac{2d}{d+1}}
 \right)^{\frac{d+1}{2d}}
 \le B_d\|T\|
\]
for every $d$-linear form $T:(\mathbb K^n)^d\to\mathbb K$, where $\mathbb K^n$ is equipped with the supremum norm and
\[
 \|T\|:=\sup\left\{
 |T(x^{(1)},\ldots,x^{(d)})|:
 \|x^{(s)}\|_\infty\le1\ \text{ for }\ 1\le s\le d
 \right\}.
\]
The exponent $2d/(d+1)$ is optimal. The dimension-free character of this estimate is the feature that has made the Bohnenblust--Hille inequality a recurring tool in the study of multilinear forms, homogeneous polynomials, Dirichlet series, and coefficient summability.

Over the decades, the inequality has developed in several directions. Mixed-norm and multiple-summing formulations revealed its connections with the geometry of Banach spaces \cite{Blei,DefantPopa,DefantSevilla}, while its polynomial form became a central ingredient in the determination of the asymptotic order of the Bohr radius \cite{BayartEtAl,DefantEtAl}. More recently, Bohnenblust--Hille phenomena have appeared for prescribed supports \cite{Bayart}, Boolean cubes \cite{DefantMastylo}, cyclic groups \cite{SloteVolbergZhang}, and in noncommutative settings \cite{VolbergZhang}. Related operator-valued and noncommutative Bohr phenomena were developed by Popescu for free holomorphic functions on polyballs \cite{Popescu2019}. Despite the variety of these developments, the underlying question remains the same: how strongly can the size of a coefficient tensor be controlled by a norm of the object that generates it, uniformly in the ambient dimension?

A different and particularly rigid form of this question was introduced by Arunachalam, Dutt, Escudero Guti\'errez and Palazuelos \cite{ADGP}. They replace the classical norm on the right-hand side by the completely bounded norm, a stronger norm adapted to operator-valued tests, and prove
\[
 \|T\|_{\frac{2d}{d+1}}\le \|T\|_{\mathrm{cb}}
\]
with constant exactly one \cite[Theorem~3.1]{ADGP}. The constant is sharp and the exponent remains optimal \cite[Section~3.3 and Theorem~3.7]{ADGP}. Thus the constant-one phenomenon belongs to this completely bounded formulation, rather than to the classical inequality. Escudero Guti\'errez, Saucedo and Palazuelos \cite[Corollary~3.3]{ESP2026} have extended the constant-one inequality to real block-multilinear polynomials without a homogeneity assumption. Once the optimal constant is known exactly, a different problem becomes natural: what does equality force?

The examples underlying the sharpness of the critical exponent already have a special algebraic structure: Walsh--Hadamard matrices produce tensors for which equality occurs. We show that this structure is exhaustive.

Every nonzero extremizer, over both $\mathbb R$ and $\mathbb C$, is supported on an equal-sided Cartesian box and factors there through a chain of Hadamard matrices. Thus equality determines both the support and the coefficient structure. For $d=2$, the nonzero coefficient matrix is a scalar multiple of a Hadamard matrix; in higher degree, the corresponding object is a Hadamard chain.

There is also a dimension-free quantitative version. A normalized form with deficit $\varepsilon$ lies within $C_d\sqrt\varepsilon$, in the critical coefficient norm, of both a flat unimodular path and a normalized unitary chain on the same equal-sided Cartesian box. Their corresponding edges satisfy a matching normalized Frobenius estimate. The exponent $1/2$ is optimal.

The proof combines spectral tests for conditional slices with a quantitative form of Blei's inequality. These estimates concentrate the coefficient magnitudes on a balanced cube and make the phases compatible along a path. A bilinear nuclear-norm estimate then yields the unitary approximation without loss of the square-root rate.

Section~2 fixes the notation and states the three main results. Sections~3--6 contain the structural and matrix estimates; the remaining sections prove the classification, stability, and optimality statements.

\section{Framework and main results}\label{sec:framework-main}

Fix $d\ge2$, $n\ge1$, and $\mathbb K\in\{\mathbb R,\mathbb C\}$. For any positive integer $k$, write $[k]=\{1,\ldots,k\}$. Let $e_1,\ldots,e_n$ be the canonical basis of $\mathbb K^n$, and let
\[
 T:(\mathbb K^n)^d\longrightarrow\mathbb K
\]
be a $d$-linear form. For $(i_1,\ldots,i_d)\in[n]^d$, the scalar
{$T(e_{i_1},\ldots,e_{i_d})$ is the corresponding coefficient. We identify $T$ with this array when taking coefficient norms or restrictions, and call its $s$th argument the $s$th slot. Empty sums are zero and empty products are one.}

All Hilbert spaces below are over $\mathbb K$, and their inner products are linear in the second variable. For such a Hilbert space $\mathcal H$, let $\mathcal B(\mathcal H)$ denote its bounded linear operators, with norm
\[
 \|X\|_{\mathcal B(\mathcal H)}
 :=\sup_{\|\xi\|_{\mathcal H}\le1}\|X\xi\|_{\mathcal H}.
\]
For each slot $s\in[d]$ and index $j\in[n]$, choose a contraction
\[
 X_j^s:\mathcal H\longrightarrow\mathcal H,
 \qquad \|X_j^s\|_{\mathcal B(\mathcal H)}\le1.
\]
All operators in a test act on the same Hilbert space and need not commute. Products denote composition in the displayed order:
\begin{equation}\label{eq:operator-composition-intro}
 X_{i_1}^1X_{i_2}^2\cdots X_{i_d}^d
 :=X_{i_1}^1\circ X_{i_2}^2\circ\cdots\circ X_{i_d}^d
 \in\mathcal B(\mathcal H).
\end{equation}
\begin{definition}\label{def:cb-norm-intro}
The completely bounded norm of $T$ is
\begin{equation}\label{eq:cb-def-intro}
 \|T\|_{\mathrm{cb}}
 :=\sup_{\mathcal H}\;
 \sup_{\substack{X_j^s\in\mathcal B(\mathcal H)\\
                  \|X_j^s\|_{\mathcal B(\mathcal H)}\le1\
                  \ (s\in[d],\ j\in[n])}}
 \left\|
 \sum_{i_1,\ldots,i_d=1}^n
 T(e_{i_1},\ldots,e_{i_d})
 X_{i_1}^1\cdots X_{i_d}^d
 \right\|_{\mathcal B(\mathcal H)}.
\end{equation}
The first supremum ranges over all Hilbert spaces $\mathcal H$ over $\mathbb K$; for each $\mathcal H$, the second ranges over all families $(X_j^s)_{s\in[d],\,j\in[n]}$ satisfying $\|X_j^s\|_{\mathcal B(\mathcal H)}\le1$.
\end{definition}

For $d\ge2$ put
\[
 q=q_d:=\frac{2d}{d+1},
\]
and define the coefficient norm
\begin{equation}\label{eq:coefficient-q-intro}
 \|T\|_q
 :=\left(\sum_{i_1,\ldots,i_d=1}^n
 |T(e_{i_1},\ldots,e_{i_d})|^q\right)^{1/q}.
\end{equation}

The completely bounded Bohnenblust--Hille inequality \cite[Theorem~3.1]{ADGP} is
\begin{equation}\label{eq:cbBH-intro}
 \|T\|_q\le\|T\|_{\mathrm{cb}}.
\end{equation}
The constant one is optimal \cite[Section~3.3]{ADGP}, and the exponent $q_d$ is optimal \cite[Theorem~3.7]{ADGP}. A nonzero equality case in \eqref{eq:cbBH-intro} is an \emph{extremizer}. When $\|T\|_{\mathrm{cb}}=1$, we call $1-\|T\|_q$ the \emph{deficit}.

Constants $C_d>0$ and smallness thresholds $\varepsilon_d>0$ depend only on $d$; $C_d$ may increase and $\varepsilon_d$ may decrease between estimates. We write $O_d(t)$ for a scalar quantity bounded in absolute value by $C_d|t|$ in the stated range.

The equality cases are built from Hadamard matrices: square matrices $H$ of order $m$ over $\mathbb K$ satisfying
\[
 |H_{ab}|=1,\qquad HH^*=mI.
\]
Here $H^*$ is the transpose over $\mathbb R$ and the adjoint over $\mathbb C$; see \cite{TadejZyczkowski} for background. If the rows and columns are indexed by $m$-element sets $I$ and $J$, then $H(i,j)$ denotes the entry at $(i,j)\in I\times J$. For sets $J_1,\ldots,J_d$ of common cardinality $m$ and matrices $G_r$ indexed by $J_r\times J_{r+1}$, we form the product of scalar entries
\begin{equation}\label{eq:scalar-path-product-intro}
 \prod_{r=1}^{d-1}G_r(i_r,i_{r+1})
 :=G_1(i_1,i_2)G_2(i_2,i_3)\cdots G_{d-1}(i_{d-1},i_d).
\end{equation}
{}

{A matrix is \emph{unimodular} if all entries have modulus one; a square matrix $U$ is \emph{unitary} if $UU^*=I$. Over $\mathbb R$, the latter means orthogonal.} For any finite scalar array $A=(A_\alpha)$, including matrices and coefficient tensors, write
\[
 \|A\|_F:=\left(\sum_\alpha |A_\alpha|^2\right)^{1/2}
\]
{for its Frobenius norm.

On $J=J_1\times\cdots\times J_d$, a \emph{unimodular path} is a function $Z(i_1,\ldots,i_d)=\prod_{r=1}^{d-1}H_r(i_r,i_{r+1})$ with unimodular \emph{edges} $H_r:J_r\times J_{r+1}\to\mathbb K$. A tensor is \emph{flat on $J$} if it has constant coefficient modulus there and vanishes elsewhere. If every $J_s$ has cardinality $m$, we call $J$ a cube of side $m$ and normalize the flat path as $m^{-(d+1)/2}Z$, extended by zero. A unimodular path on a cube is a \emph{Hadamard chain} when each edge matrix $H_r$ is Hadamard.} A \emph{normalized unitary chain} has coefficients $m^{-1}\prod_{r=1}^{d-1}U_r(i_r,i_{r+1})$ on $J$, zero elsewhere, with unitary edges $U_r$.

\begin{maintheorem}\label{thm:classification}
Let $T:(\mathbb K^n)^d\to\mathbb K$ be a nonzero $d$-linear form. Equality holds in \eqref{eq:cbBH-intro} if and only if there are an integer $1\le m\le n$, sets $J_1,\ldots,J_d\subseteq[n]$ satisfying $|J_1|=\cdots=|J_d|=m$, a scalar $\lambda\ne0$, and Hadamard matrices $G_r$ indexed by $J_r\times J_{r+1}$ such that
\begin{equation}\label{eq:classification}
 T(e_{i_1},\ldots,e_{i_d})
 =\begin{cases}
 \displaystyle\lambda\prod_{r=1}^{d-1}G_r(i_r,i_{r+1}),
       &(i_1,\ldots,i_d)\in J_1\times\cdots\times J_d,\\[4pt]
 0,&\text{otherwise.}
 \end{cases}
\end{equation}
\end{maintheorem}

{The theorem determines both the support and the factorization of every extremizer. Its Hadamard edges are real, with entries in $\{\pm1\}$, over $\mathbb R$, and complex over $\mathbb C$. For $d=2$, the active coefficient matrix is a scalar multiple of a Hadamard matrix.}

\begin{maintheorem}\label{thm:flat-stability}
For each fixed $d\ge2$ there are constants $C_d<\infty$ and $\varepsilon_d>0$ such that, whenever
\[
 \|T\|_{\mathrm{cb}}=1,\qquad \|T\|_q\ge1-\varepsilon,
 \qquad 0\le\varepsilon\le\varepsilon_d,
\]
there are an integer $1\le m\le n$, sets
$J_1,\ldots,J_d\subseteq[n]$ with $|J_1|=\cdots=|J_d|=m$, and
unimodular matrices $H_r$, indexed by $J_r\times J_{r+1}$ for
$r=1,\ldots,d-1$. Put $J:=J_1\times\cdots\times J_d$ and define
\[
 P(e_{i_1},\ldots,e_{i_d})
 =m^{-(d+1)/2}\prod_{r=1}^{d-1}H_r(i_r,i_{r+1})
 \quad ((i_1,\ldots,i_d)\in J),
\]
with $P(e_{i_1},\ldots,e_{i_d})=0$ outside $J$. Then
\[
 \|T-P\|_q\le C_d\sqrt\varepsilon,
 \qquad \|P\|_q=1,
\]
and
\begin{equation}\label{eq:main-tail}
 {\sum_{(i_1,\ldots,i_d)\in[n]^d\setminus J}}|T(e_{i_1},\ldots,e_{i_d})|^q\le C_d\varepsilon.
\end{equation}
\end{maintheorem}

\begin{maintheorem}\label{thm:unitary-stability}
Under the hypotheses of Theorem~\ref{thm:flat-stability}, the sets
$J_1,\ldots,J_d$ and the unimodular matrices $H_r$ may be chosen so
that there are unitary matrices $U_r$ of order $m$,
$r=1,\ldots,d-1$, for which the tensor
\[
 Q(e_{i_1},\ldots,e_{i_d})
 =\frac1m\prod_{r=1}^{d-1}U_r(i_r,i_{r+1})
 \quad ((i_1,\ldots,i_d)\in J),
\]
extended by zero outside $J$, satisfies
\[
 \|T-Q\|_q\le C_d\sqrt\varepsilon,
 \qquad \|Q\|_{\mathrm{cb}}=1,
 \qquad \|Q\|_F=m^{-1/2},
\]
and
\begin{equation}\label{eq:main-sqrt}
 \max_{1\le r\le d-1}
 \frac{\|H_r-\sqrt m\,U_r\|_F}{m}
 \le C_d\sqrt\varepsilon.
\end{equation}
Over $\mathbb R$, every entry of $H_r$ belongs to $\{\pm1\}$ and the $U_r$ are orthogonal.
\end{maintheorem}

Theorems~\ref{thm:flat-stability} and~\ref{thm:unitary-stability} retain, respectively, the constant coefficient modulus and the unitary edges of the exact models. At zero deficit both approximants equal $T$; for positive deficit neither is required to be an exact extremizer.

\section{Preliminaries and spectral estimates}\label{sec:spectral-input}

{
For a nonempty finite set $E$ and $1\le p<\infty$, the counting and probability norms are
\begin{equation}\label{eq:counting-probability-norms}
 \|f\|_{\ell_p(E)}=\left(\sum_{x\in E}|f(x)|^p\right)^{1/p},
 \qquad
 \|f\|_{L^p(E)}=\left(\frac1{|E|}\sum_{x\in E}|f(x)|^p\right)^{1/p}.
\end{equation}
Both norms at $p=\infty$ are $\max_{x\in E}|f(x)|$.
We write $\mathbb E_E f=|E|^{-1}\sum_{x\in E}f(x)$ for the uniform average.
The symbols $L^p$ and $\mathbb E$ refer to uniform product measure unless another probability law is specified.

For $J=J_1\times\cdots\times J_d\subseteq[n]^d$, define
\[
 (T|_J)(i_1,\ldots,i_d):=T(e_{i_1},\ldots,e_{i_d})
 \qquad ((i_1,\ldots,i_d)\in J),
\]
and write $T_J:=T|_J$. Its extension by zero is the array
\[
 (T\mathbf1_J)(i_1,\ldots,i_d)
 :=\begin{cases}
 T(e_{i_1},\ldots,e_{i_d}),&(i_1,\ldots,i_d)\in J,\\
 0,&(i_1,\ldots,i_d)\notin J.
 \end{cases}
\]
The support is $\operatorname{supp}T=\{i\in[n]^d:T(e_{i_1},\ldots,e_{i_d})\ne0\}$.}

Fix $s\in[d]$ and $j\in[n]$. The \emph{$(s,j)$-slice} of $T$ is the
coefficient array obtained by fixing $i_s=j$. Its matrix $M_{s,j}$ has
row index $p=(i_1,\ldots,i_{s-1})\in[n]^{s-1}$ and column index
$u=(i_{s+1},\ldots,i_d)\in[n]^{d-s}$, with entries
\begin{equation}\label{eq:slice-matrix-definition}
 M_{s,j}(p,u)
 =T(e_{i_1},\ldots,e_{i_{s-1}},e_j,e_{i_{s+1}},\ldots,e_{i_d}).
\end{equation}
At the endpoints, $[n]^0=\{\varnothing\}$: $M_{1,j}$ is a row matrix and
$M_{d,j}$ is a column matrix. If $j\in J_s$, then $M^J_{s,j}$ denotes
the restriction to rows in $\prod_{t=1}^{s-1}J_t$ and columns in
$\prod_{t=s+1}^dJ_t$.

Matrix norms use counting measure unless otherwise specified. For $M=(M(a,b))_{a\in A,b\in B}$, the operator norm is the norm of
$f\mapsto(\sum_{b\in B}M(a,b)f(b))_{a\in A}$ from $\ell_2(B)$ to $\ell_2(A)$.
Its squared Frobenius norm is $\|M\|_F^2=\sum_{a\in A}\sum_{b\in B}|M(a,b)|^2$.
For a matrix or array $M$ and a nonempty set $\mathcal R$ in the same normed
space, $\operatorname{dist}(M,\mathcal R)=\inf_{R\in\mathcal R}\|M-R\|$;
the subscript specifies the norm.

Let $\sigma_1(M)\ge\sigma_2(M)\ge\cdots\ge0$ be the singular values of $M$,
extended by zeros. Thus $\|M\|_{\mathrm{op}}=\sigma_1(M)$ and
$\|M\|_F^2=\sum_{k\ge1}\sigma_k(M)^2$. The distance from $M$ to the
matrices of rank at most one is
\begin{equation}\label{eq:rank-one-tail-definition}
 \operatorname{tail}_2(M)^2
 :=\operatorname{dist}_F(M,\{R:\operatorname{rank}R\le1\})^2
 =\sum_{k\ge2}\sigma_k(M)^2.
\end{equation}
The trace, denoted by $\Tr$, is the sum of the diagonal entries. The trace
(or nuclear) norm is
\begin{equation}\label{eq:trace-norm-definition}
 \|M\|_{S_1}:=\Tr\bigl((M^*M)^{1/2}\bigr)=\sum_{k\ge1}\sigma_k(M).
\end{equation}

For nonempty finite sets $X_1,\ldots,X_d$ and a function
$F:X_1\times\cdots\times X_d\to\mathbb K$, fix $s\in[d]$ and $j\in X_s$.
For
\[
 p=(x_1,\ldots,x_{s-1})\in\prod_{t<s}X_t,
 \qquad
 u=(x_{s+1},\ldots,x_d)\in\prod_{t>s}X_t,
\]
with the empty Cartesian products interpreted as $\{\varnothing\}$, define the conditional matrix
\begin{equation}\label{eq:conditional-matrix-F}
 M_{s,j}(F)(p,u)
 :=F(x_1,\ldots,x_{s-1},j,x_{s+1},\ldots,x_d).
\end{equation}
The notation $\mathsf M_{s,j}(F)$ denotes this kernel interpreted as
an operator between the corresponding probability $L^2$ spaces. More generally, a kernel $M$ on $A\times B$ acts by
\begin{equation}\label{eq:probability-kernel-convention}
 (Mf)(a)=\frac1{|B|}\sum_{b\in B}M(a,b)f(b),
 \qquad
 \|M\|_{\mathrm{HS},\mathrm{prob}}
 =(|A||B|)^{-1/2}\|M\|_F.
\end{equation}
Here $\mathrm{HS}$ denotes the Hilbert--Schmidt norm: the square root of the sum of the squared singular values. We omit $\mathrm{prob}$ when the measure is clear. Thus a conditional slice on a box of side $m$ satisfies
\[
 \|\mathsf M_{s,j}(F)\|_{\mathrm{HS},\mathrm{prob}}
 =m^{-(d-1)/2}\|M_{s,j}(F)\|_F.
\]

For the full slices of $T$, put
\begin{equation}\label{eq:slice-quantities}
 h_{s,j}=\|M_{s,j}\|_F,\qquad A_s=\sum_{j=1}^n h_{s,j},\qquad
 \Gamma_s=\sum_{\substack{1\le j\le n\\ h_{s,j}>0}}
 \frac{h_{s,j}^2}{\|M_{s,j}\|_{\mathrm{op}}}.
\end{equation}
We write $A_s(T)$ and $\Gamma_s(T)$ when the form must be specified. Quotients corresponding to zero slices are defined as zero.

The following form of Blei's mixed-norm inequality is the case $J=d$, $K=1$ of \cite[Lemma~5.3]{Blei}.

{\begin{lemma}[{Blei's inequality}]
\label{lem:blei}
Let $a=(a_{i_1,\ldots,i_d})$ be a nonnegative array on $[n]^d$, let
\[
q=\frac{2d}{d+1},
\qquad
R_s(j)=\left(\sum_{\substack{i\in[n]^d\\ i_s=j}}a_i^2\right)^{1/2},
\qquad
A_s=\sum_{j=1}^nR_s(j).
\]
Then
\begin{equation}
\label{eq:blei}
\|a\|_{\ell_q([n]^d)}
\le
\left(\prod_{s=1}^d A_s\right)^{1/d}
\le
\frac1d\sum_{s=1}^dA_s.
\end{equation}
\end{lemma}

\begin{proof}
The first inequality is \cite[Lemma~5.3, with $J=d$ and $K=1$]{Blei}; the second follows from the arithmetic--geometric mean inequality.
\end{proof}}

Over $\mathbb R$, the estimate for $\Gamma_s$ below follows from the stronger trace-norm slice bound in \cite[Remark~3.2]{ESP2026}. The direct operator construction used here applies over both fields.

\begin{lemma}\label{lem:spectral-input}
For every slot $s$, one has $A_s\le {\Gamma_s}\le\|T\|_{\mathrm{cb}}$.
\end{lemma}
\begin{proof}
For every matrix $M$,
\[
 \|M\|_{\mathrm{op}}\le \|M\|_F,
\]
and therefore, term by term,
\[
 h_{s,j}
 \le \frac{h_{s,j}^2}{\|M_{s,j}\|_{\mathrm{op}}}
 \qquad (h_{s,j}>0).
\]
Summing in $j$ gives $A_s\le {\Gamma_s}$.

{For ${\Gamma_s}\le\|T\|_{\mathrm{cb}}$, we adapt \cite[Lemma~3.3]{ADGP} by normalizing its connecting block with $\|M_{s,j}\|_{\mathrm{op}}$ in place of $h_{s,j}$.} Fix $s$. For $0\le r\le d-s$, let $E_r$ be a copy of $\ell_2([n]^r)$ with orthonormal basis
\[
 \{e_{(i_{d-r+1},\ldots,i_d)}:(i_{d-r+1},\ldots,i_d)\in[n]^r\},
\]
where $E_0=\operatorname{span}\{e_\varnothing\}$. For $0\le t\le s-1$, let $F_t$ be a copy of $\ell_2([n]^t)$ with orthonormal basis
\[
 \{f_{(i_1,\ldots,i_t)}:(i_1,\ldots,i_t)\in[n]^t\},
\]
where $F_0=\operatorname{span}\{f_\varnothing\}$. Set
\[
 \mathcal H=\left(\bigoplus_{r=0}^{d-s}E_r\right)
 \oplus\left(\bigoplus_{t=0}^{s-1}F_t\right).
\]

Define $Z_j\in\mathcal B(\mathcal H)$, for each $j\in[n]$, by its blocks. On the $E$-levels, for $0\le r<d-s$, put
\[
 Z_j e_{(i_{d-r+1},\ldots,i_d)}
 =e_{(j,i_{d-r+1},\ldots,i_d)}.
\]
Thus, starting from $e_\varnothing$, the operators on the right of the $s$th slot produce
\[
 Z_{i_{s+1}}\cdots Z_{i_d}e_\varnothing
 =e_{(i_{s+1},\ldots,i_d)}.
\]
The bridge from $E_{d-s}$ to $F_{s-1}$ connects the chains. For $M_{s,j}\ne0$, define
\[
 Z_j e_{(i_{s+1},\ldots,i_d)}
 =\frac1{\|M_{s,j}\|_{\mathrm{op}}}
 \sum_{i_1,\ldots,i_{s-1}=1}^n
 \overline{T(e_{i_1},\ldots,e_{i_{s-1}},e_j,
 e_{i_{s+1}},\ldots,e_{i_d})}
 f_{(i_1,\ldots,i_{s-1})},
\]
and set this bridge block equal to zero when $M_{s,j}=0$. Finally, on the $F$-levels, for $1\le t\le s-1$, put
\[
 Z_j f_{(i_1,\ldots,i_t)}
 =\mathbf 1_{\{i_t=j\}}f_{(i_1,\ldots,i_{t-1})}.
\]
All unspecified blocks are zero.

The $E$-blocks are isometries, the $F$-blocks are coordinate partial isometries, and the normalized bridge is a contraction. For fixed $j$, distinct blocks have orthogonal source levels and orthogonal target levels. Thus, decomposing $x=\sum_Lx_L$ by levels, we obtain
\[
 \|Z_jx\|^2
 =\sum_L\|Z_jx_L\|^2
 \le\sum_L\|x_L\|^2
 =\|x\|^2.
\]
Hence every $Z_j$ is a contraction.

{Starting at $e_\varnothing$, we have}
\[
 Z_{i_{s+1}}\cdots Z_{i_d}e_\varnothing
 =e_{(i_{s+1},\ldots,i_d)}.
\]
The bridge operator $Z_{i_s}$ then gives
\[
 \frac1{\|M_{s,i_s}\|_{\mathrm{op}}}
 \sum_{j_1,\ldots,j_{s-1}=1}^n
 \overline{T(e_{j_1},\ldots,e_{j_{s-1}},e_{i_s},
 e_{i_{s+1}},\ldots,e_{i_d})}
 f_{(j_1,\ldots,j_{s-1})}.
\]
For $1\le t\le s-1$, the definition of the $F$-level block gives
\[
 Z_{i_t}f_{(j_1,\ldots,j_t)}
 =\mathbf 1_{\{j_t=i_t\}}f_{(j_1,\ldots,j_{t-1})}.
\]
Hence, for every $(j_1,\ldots,j_{s-1})$,
\begin{align*}
 Z_{i_1}\cdots Z_{i_{s-1}}f_{(j_1,\ldots,j_{s-1})}
 &=\left(\prod_{t=1}^{s-1}\mathbf 1_{\{j_t=i_t\}}\right)f_\varnothing.
\end{align*}
This identity selects the term $(j_1,\ldots,j_{s-1})=(i_1,\ldots,i_{s-1})$ from the bridge sum. Hence
\[
 \left\langle f_\varnothing,
 Z_{i_1}\cdots Z_{i_d}e_\varnothing\right\rangle
 =\frac{\overline{T(e_{i_1},\ldots,e_{i_d})}}
 {\|M_{s,i_s}\|_{\mathrm{op}}},
\]
with value zero when $M_{s,i_s}=0$. Therefore
\begin{align*}
 \|T\|_{\mathrm{cb}}
 &\ge
 \left|
 \left\langle f_\varnothing,
 \sum_{i_1,\ldots,i_d=1}^n
 T(e_{i_1},\ldots,e_{i_d})
 Z_{i_1}\cdots Z_{i_d}e_\varnothing
 \right\rangle\right|\\
 &=\sum_{i_1,\ldots,i_d=1}^n
 \frac{|T(e_{i_1},\ldots,e_{i_d})|^2}
 {\|M_{s,i_s}\|_{\mathrm{op}}}\\
 &=\sum_{\substack{1\le j\le n\\h_{s,j}>0}}
 \frac{\|M_{s,j}\|_F^2}{\|M_{s,j}\|_{\mathrm{op}}}
 ={\Gamma_s}.
\end{align*}
For $s=1$, the $F$-chain is absent and the bridge lands in $F_0$; for $s=d$, the $E$-chain is absent and the bridge starts in $E_0$. The same computation applies.
\end{proof}

{\begin{corollary}
For $T\ne0$, define the averaged slice defect
\[
 \delta_{\mathrm{sl}}(T)
 :=1-\frac1{d\|T\|_{\mathrm{cb}}}\sum_{s=1}^d A_s(T).
\]
Then
\[
 \sum_{s=1}^d{\sum_{\substack{1\le j\le n\\M_{s,j}\ne0}}}
 \frac{\sum_{k\ge2}\sigma_k(M_{s,j})^2}{\|M_{s,j}\|_F}
 \le 2d\,\delta_{\mathrm{sl}}(T)\|T\|_{\mathrm{cb}}.
\]
In particular, if $\|T\|_q\ge(1-\varepsilon)\|T\|_{\mathrm{cb}}$, the right-hand side is at most $2d\varepsilon\|T\|_{\mathrm{cb}}$.
\end{corollary}
\begin{proof}
For a nonzero matrix $M$, with $h=\|M\|_F$ and $r=\|M\|_{\mathrm{op}}$,
\[
 \frac{h^2}{r}-h
 \ge\frac{h^2-r^2}{2h}
 =\frac1{2h}\sum_{k\ge2}\sigma_k(M)^2.
\]
Summing and using ${\Gamma_s}\le\|T\|_{\mathrm{cb}}$ proves the first assertion. Lemma~\ref{lem:blei}, applied to the coefficients of $T$, gives
$\|T\|_q\le({\prod_{s=1}^d A_s})^{1/d}\le d^{-1}{\sum_{s=1}^d A_s}$, and hence $\delta_{\mathrm{sl}}(T)\le\varepsilon$.
\end{proof}}

\begin{lemma}\label{lem:unitary-chain}
{Let $J_1,\ldots,J_d$ have common cardinality $m$, and let $U_r$ be unitary,
with rows indexed by $J_r$ and columns indexed by $J_{r+1}$, for
$1\le r\le d-1$. The tensor with coefficients
$Q(e_{i_1},\ldots,e_{i_d})=m^{-1}\prod_{r=1}^{d-1}U_r(i_r,i_{r+1})$
on $J_1\times\cdots\times J_d$, and zero outside, satisfies}
\[
 \|Q\|_{\mathrm{cb}}=1,\qquad \|Q\|_F=m^{-1/2}.
\]
Its active first-slot and last-slot slice Frobenius norms are all $1/m$.
\end{lemma}
\begin{proof}
{Identify each $J_s$ separately with $[m]$.}
Let $X_j^s\in\mathcal B(\mathcal H)$ be contractions. Define
\[
 L=m^{-1/2}[X_1^1\ \cdots\ X_m^1]:\mathcal H^m\to\mathcal H,
 \qquad
 R=m^{-1/2}[X_1^d\ \cdots\ X_m^d]^{\mathsf T}:\mathcal H\to\mathcal H^m,
\]
and, for $2\le s\le d-1$,
\[
 D_s=\operatorname{diag}(X_1^s,\ldots,X_m^s):\mathcal H^m\to\mathcal H^m.
\]
For $(\xi_1,\ldots,\xi_m)\in\mathcal H^m$,
\begin{align*}
 \|L(\xi_1,\ldots,\xi_m)\|^2
 &=\frac1m\left\|\sum_{j=1}^mX_j^1\xi_j\right\|^2\\
 &\le\frac1m\left(\sum_{j=1}^m\|\xi_j\|\right)^2\\
 &\le\sum_{j=1}^m\|\xi_j\|^2,
\end{align*}
so $\|L\|\le1$. For $\xi\in\mathcal H$,
\[
 \|R\xi\|^2
 =\frac1m\sum_{j=1}^m\|X_j^d\xi\|^2
 \le\|\xi\|^2,
\]
and
\[
 \|D_s(\xi_1,\ldots,\xi_m)\|^2
 =\sum_{j=1}^m\|X_j^s\xi_j\|^2
 \le\sum_{j=1}^m\|\xi_j\|^2.
\]
{Hence $L,R,D_s$ are contractions. The block operator
$U_r\otimes I_{\mathcal H}$ on $\mathcal H^m$ is defined by
\[
 (U_r\otimes I_{\mathcal H})(\xi_1,\ldots,\xi_m)
 =\left(\sum_{b=1}^m U_r(a,b)\xi_b\right)_{a=1}^m.
\]
This operator is unitary because $U_r^*U_r=U_rU_r^*=I_m$.}

Expanding the block product, interpreted as $L(U_1\otimes I)R$ when $d=2$, gives for every $\xi\in\mathcal H$
\begin{align}
 &L(U_1\otimes I)D_2(U_2\otimes I)\cdots
 D_{d-1}(U_{d-1}\otimes I)R\,\xi\notag\\
 &\quad=\sum_{i_1,\ldots,i_d=1}^m
 \frac1m\left(\prod_{r=1}^{d-1}U_r(i_r,i_{r+1})\right)
 X_{i_1}^1\cdots X_{i_d}^d\xi\notag\\
 &\quad=\sum_{i_1,\ldots,i_d=1}^m
 Q(e_{i_1},\ldots,e_{i_d})
 X_{i_1}^1\cdots X_{i_d}^d\xi.
 \label{eq:unitary-chain-factorization}
\end{align}
Every factor on the first line of \eqref{eq:unitary-chain-factorization} is a contraction, so
\[
 \left\|
 \sum_{i_1,\ldots,i_d=1}^m
 Q(e_{i_1},\ldots,e_{i_d})X_{i_1}^1\cdots X_{i_d}^d
 \right\|\le1.
\]
Taking the supremum over all contraction families gives $\|Q\|_{\mathrm{cb}}\le1$.

Fix $i_1$. By the definition of $Q$,
\begin{align*}
 &\sum_{i_2,\ldots,i_d=1}^m
 |Q(e_{i_1},\ldots,e_{i_d})|^2\\
 &\quad=\frac1{m^2}
 \sum_{i_2=1}^m|U_1(i_1,i_2)|^2
 \sum_{i_3=1}^m|U_2(i_2,i_3)|^2\cdots
 \sum_{i_d=1}^m|U_{d-1}(i_{d-1},i_d)|^2\\
 &\quad=\frac1{m^2},
\end{align*}
since each row of $U_r$ has norm one. Thus every first-slot slice has Frobenius norm $1/m$, and
\[
 A_1(Q)=\sum_{i_1=1}^m\frac1m=1.
\]
Lemma~\ref{lem:spectral-input} gives
\[
 1=A_1(Q)\le\|Q\|_{\mathrm{cb}}.
\]
Combining the two bounds yields
\[
 \|Q\|_{\mathrm{cb}}=1.
\]
Furthermore,
\[
 \|Q\|_F^2
 =\sum_{i_1=1}^m
 \sum_{i_2,\ldots,i_d=1}^m
 |Q(e_{i_1},\ldots,e_{i_d})|^2
 =m\frac1{m^2}=\frac1m,
\]
so $\|Q\|_F=m^{-1/2}$. Summing in reverse order and using the unit column norms of $U_r$ gives, for each fixed $i_d$,
\[
 \sum_{i_1,\ldots,i_{d-1}=1}^m
 |Q(e_{i_1},\ldots,e_{i_d})|^2=\frac1{m^2},
\]
which proves the last-slot assertion.
\end{proof}

\begin{remark}[Classical norm of complex unitary chains]
\label{rem:classical-norm-complex-hadamard-chains}

Assume that $\mathbb K=\mathbb C$. Let $Q$ be a normalized unitary chain on $J_1\times\cdots\times J_d$, where $|J_s|=m$, with unitary edges $U_r$. Then
\[
 \|Q\|=\|Q\|_{\mathrm{cb}}=1.
\]
If each $H_r=\sqrt m\,U_r$ is Hadamard, then also $\|Q\|_q=1$: its $m^d$ nonzero coefficients have modulus $m^{-(d+1)/2}$.

Lemma~\ref{lem:unitary-chain} gives $\|Q\|_{\mathrm{cb}}=1$, and scalar contractions are admissible in the completely bounded norm, so $\|Q\|\le1$.

Conversely, write $\mathbb T=\{z\in\mathbb C:|z|=1\}$ and let
$\mathbf 1=(1,\ldots,1)^{\mathsf T}$. By the normal-form theorem of Idel and
Wolf \cite[Theorem~2]{IdelWolf}, for every unitary $U$ there are diagonal
unitaries $L,R$ such that $A=LUR$ has all row sums equal to one. Hence
$A\mathbf1=\mathbf1$, and therefore, with
$\eta:=R\mathbf1\in\mathbb T^m$,
\[
 U\eta=L^*\mathbf1\in\mathbb T^m.
\]
Thus every unitary matrix has a biunimodular vector. For
$\xi\in\mathbb T^m$, write
\[
 \operatorname{Diag}(\xi):=\operatorname{diag}(\xi(1),\ldots,\xi(m)).
\]
We choose the unimodular vectors from right to left. First choose \(\xi_d\) so that
\(U_{d-1}\xi_d\in\mathbb T^m\). At each preceding step, if
\(y\in\mathbb T^m\) is the unimodular vector already produced to the
right of \(U_r\), choose a biunimodular vector \(\eta\) for \(U_r\) and
set \(\xi_{r+1}=\eta\,\overline y\) coordinatewise. Then
\(\operatorname{Diag}(\xi_{r+1})y=\eta\), and hence
\(U_r\operatorname{Diag}(\xi_{r+1})y\in\mathbb T^m\). Continuing in this way gives
\[
\zeta:=
U_1\operatorname{Diag}(\xi_2)U_2\operatorname{Diag}(\xi_3)\cdots
\operatorname{Diag}(\xi_{d-1})U_{d-1}\xi_d
\in\mathbb T^m.
\]
Taking \(\xi_1=\overline{\zeta}\), we obtain
\[
\sum_{i_1,\ldots,i_d=1}^m
Q(e_{i_1},\ldots,e_{i_d})
\prod_{s=1}^d\xi_s(i_s)
=
\frac1m\sum_{j=1}^m|\zeta(j)|^2
=1.
\]
Thus \(\|Q\|\ge1\), proving the claim.
\end{remark}
\section{Reduction to a balanced cube}
\label{sharp-sec-cube}

{ \subsection{\texorpdfstring{{Normalization on a cube}}{Normalization on a cube}}
\label{subsec:flat-scale-bookkeeping}

Let $q=2d/(d+1)$. If
\[
J=J_1\times\cdots\times J_d,
\qquad
|J_1|=\cdots=|J_d|=m,
\qquad
N=|J|=m^d,
\]
then a tensor supported on $J$, of coefficient $\ell_q$ norm one and with constant modulus on $J$, has coefficient modulus
\[
\rho=N^{-1/q}=m^{-(d+1)/2}.
\]
Thus
\[
\rho N^{1/q}=1.
\]
If $T_J$ denotes the restriction of $T$ to $J$ and $F=\rho^{-1}T_J$, with
normalized counting measure on $J$, then
\[
\|F\|_{L^q(J)}=\|T_J\|_{\ell_q(J)}.
\]
Moreover, for every conditional slice $M^J_{s,j}(T)$ of $T_J$,
\[
\|\mathsf M_{s,j}(F)\|_{\mathrm{HS},\mathrm{prob}}
=\rho^{-1}m^{-(d-1)/2}\|M^J_{s,j}(T)\|_F
=m\|M^J_{s,j}(T)\|_F.
\]}

\begin{lemma}\label{lem:q-moment-variance}
Let $1<q<2$, let $(\Omega,\mu)$ be a finite probability space, and let
$x\ge0$ satisfy $\mathbb E x=1$. Then
\[
 1\le \mathbb E x^q
 \le 1+(q-1)\mathbb E(x-1)^2.
\]
Consequently, if $\mathbb E(x-1)^2\le\eta$, then
\[
 1\le (\mathbb E x^q)^{1/q}\le 1+C_q\eta
\]
whenever $0\le\eta\le1$.
\end{lemma}

\begin{proof}
Convexity of $t\mapsto t^q$ and $\mathbb E x=1$ give
\[
 1=(\mathbb E x)^q\le\mathbb E x^q.
\]
For $x\ge0$, weighted Young's inequality gives
\[
 x^q\le(q-1)x^2+(2-q)x.
\]
Taking expectations and using
\[
 \mathbb E x^2=\mathbb E(x-1)^2+2\mathbb E x-1
 =1+\mathbb E(x-1)^2
\]
yields
\begin{align*}
 \mathbb E x^q
 &\le(q-1)\mathbb E x^2+(2-q)\mathbb E x\\
 &=(q-1)\bigl(1+\mathbb E(x-1)^2\bigr)+(2-q)\\
 &=1+(q-1)\mathbb E(x-1)^2.
\end{align*}
The last assertion follows from $(1+t)^{1/q}\le1+C_qt$ for
$0\le t\le q-1$.
\end{proof}
\medskip
\begin{lemma}
\label{sharp-lem-cube}
Fix $d\ge2$ and $q=2d/(d+1)$.  Let $a$ be a nonzero nonnegative array on
{a finite Cartesian product $\mathcal Y=Y_1\times\cdots\times Y_d$, and set
\[
 R_s(j)=\left(\sum_{\substack{i\in\mathcal Y\\i_s=j}}a_i^2\right)^{1/2},
 \qquad A_s=\sum_{j\in Y_s}R_s(j)\qquad(j\in Y_s).
\]}
Suppose that
\[
 A_s\le1\quad(1\le s\le d),\qquad
 \|a\|_{\ell_q}\ge1-\varepsilon.
\]
There are constants $C_d<\infty$ and $\varepsilon_d>0$, depending only
on $d$, with the following property.  If
$0\le\varepsilon\le\varepsilon_d$, there are nonempty sets {$J_s\subseteq Y_s$} of a
common cardinality $m$ such that, with
\[
 J=\prod_{s=1}^dJ_s,\qquad N=m^d,\qquad \rho=N^{-1/q},
\]
one has
\begin{align}
 {\sum_{i\in\mathcal Y\setminus J}}a_i^q&\le C_d\varepsilon,
 \label{sharp-eq-tail}\\
 \frac1N\sum_{i\in J}\left|\frac{a_i}{\rho}-1\right|^2
 &\le C_d\varepsilon,
 \label{sharp-eq-variance}\\
 \left|\frac1N\sum_{i\in J}\frac{a_i}{\rho}-1\right|
 +\left|\frac1N\sum_{i\in J}\left(\frac{a_i}{\rho}\right)^2-1\right|
 &\le C_d\varepsilon.
 \label{sharp-eq-moments}
\end{align}
The original full slice norms satisfy
\begin{equation}
 R_s(j)\le\frac{C_d}{m}\qquad(j\in J_s).
 \label{sharp-eq-full-slices}
\end{equation}
In particular,
\begin{equation}
 \|a-\rho\mathbf1_J\|_{\ell_q}\le C_d\sqrt\varepsilon.
 \label{sharp-eq-flat-approximation}
\end{equation}
\end{lemma}

\begin{proof}
\noindent\emph{Step 1: Hellinger control.}
{Write $k=d+1$ and $B=(\prod_{s=1}^d A_s)^{1/d}$.
All the marginal sums $A_s$ are positive. Define}
probability distributions
\[
 p_s(j)=\frac{R_s(j)}{A_s},\qquad
 \Pi(i)=\prod_{s=1}^dp_s(i_s),\qquad
 \nu_s(i)=\frac{a_i^2}{R_s(i_s)A_s},
\]
where a quotient with zero denominator is assigned the value zero.
{Set
$\Omega=\prod_{s=1}^d\{j\in Y_s:p_s(j)>0\}$.
The array $a$ vanishes outside $\Omega$, and $\Pi>0$ on $\Omega$. Indeed,}
\begin{align*}
 \Pi(i)\prod_{s=1}^d\nu_s(i)
 &=\prod_{s=1}^d\frac{R_s(i_s)}{A_s}
   \prod_{s=1}^d\frac{a_i^2}{R_s(i_s)A_s}\\
 &=\frac{a_i^{2d}}{\prod_{s=1}^d A_s^2}
 =\frac{a_i^{2d}}{B^{2d}}.
\end{align*}
Since $k=d+1$ and $q=2d/(d+1)$, it follows that
\[
 G(i):=\left(\Pi(i)\prod_{s=1}^d\nu_s(i)\right)^{1/k}
      =\frac{a_i^q}{B^q}.
\]
H\"older's inequality, applied to the $d+1$ probability distributions
$\Pi,\nu_1,\ldots,\nu_d$, gives
\[
 {\sum_{i\in\Omega}}G(i)
 \le \left({\sum_{i\in\Omega}}\Pi(i)\right)^{1/k}
      \prod_{s=1}^d\left({\sum_{i\in\Omega}}\nu_s(i)\right)^{1/k}
 =1.
\]
{Define}
\[
\delta_{\mathrm{Blei}}
:=
1-
\left(
\frac{\|a\|_{\ell_q}}{B}
\right)^q.
\]
Since \(B\le1\) and \(\|a\|_{\ell_q}\ge1-\varepsilon\), we have
\[
0\le\delta_{\mathrm{Blei}}
\le
1-(1-\varepsilon)^q
\le
q\varepsilon.
\]

For nonnegative numbers $F_0,\ldots,F_d$, the arithmetic--geometric
mean inequality applied to their square roots gives
\[
 \frac1k\sum_{r=0}^dF_r
 -\left(\prod_{r=0}^dF_r\right)^{1/k}
 \ge\frac1k\sum_{r=0}^dF_r
      -\left(\frac1k\sum_{r=0}^d\sqrt{F_r}\right)^2
 =\frac1{k^2}{\sum_{0\le r<t\le d}}(\sqrt{F_r}-\sqrt{F_t})^2.
\]
Apply this inequality pointwise to $F_0=\Pi$ and $F_s=\nu_s$, and sum.
In particular,
\begin{equation}
 \|\sqrt{\nu_s}-\sqrt\Pi\|_{\ell_2}^2\le k^2\delta_{\mathrm{Blei}}
 \qquad(1\le s\le d).
 \label{sharp-eq-hellinger}
\end{equation}

\medskip
\noindent\emph{Step 2: Concentration around a common level.}
{Under the probability law $\Pi$ on $\Omega$, the coordinate maps
$X_s(i_1,\ldots,i_d)=i_s$ are independent and have laws $p_s$. Put}
\[
 W_s=A_s^2p_s(X_s),\qquad \eta_s=\sqrt{\nu_s/\Pi}.
\]
Since $p_s(i_s)=R_s(i_s)/A_s$,
\[
 W_s=A_s^2p_s(i_s)=A_sR_s(i_s),
\]
and hence, on the positive marginal support,
\[
 \eta_s^2W_s
 =\frac{\nu_s}{\Pi}A_sR_s(i_s)
 =\frac{a_i^2}{\Pi(i)}.
\]
Thus
\[
 \eta_s\sqrt{W_s}=\frac{a_i}{\sqrt{\Pi(i)}}
 =\eta_t\sqrt{W_t}
\]
pointwise, also when $a_i=0$.
For positive $w,z$, define
\[
 d_0(w,z)=\frac{|w-z|}{w+z}
 =\tanh\!\left(\frac{|\log w-\log z|}{2}\right).
\]
This is a metric: its triangle inequality follows from that for absolute logarithmic differences, the monotonicity of $\tanh$, and $\tanh(x+y)\le\tanh x+\tanh y$ for $x,y\ge0$.

If $r,t\ge0$ and $r\sqrt w=t\sqrt z$, write their common value as
$c=r\sqrt w=t\sqrt z$. Then
\[
 (r-1)^2+(t-1)^2
 =\left(\frac{c}{\sqrt w}-1\right)^2
  +\left(\frac{c}{\sqrt z}-1\right)^2.
\]
The quadratic on the right is minimized at
\[
 c_*=\frac{\sqrt{wz}(\sqrt w+\sqrt z)}{w+z},
\]
and substitution gives
\[
 (r-1)^2+(t-1)^2
 \ge\frac{(\sqrt w-\sqrt z)^2}{w+z}
 \ge\frac12d_0(w,z)^2.
\]
Consequently, \eqref{sharp-eq-hellinger} yields
\begin{equation}
 \mathbb E_\Pi d_0(W_s,W_t)^2\le4k^2\delta_{\mathrm{Blei}}.
 \label{sharp-eq-pair-metric}
\end{equation}
For $(s,t)=(1,2)$, independence gives
\[
 \mathbb E d_0(W_1,W_2)^2
 =\mathbb E_{X_2}\!\left[
   \mathbb E_{X_1}d_0(W_1,W_2)^2\mid X_2\right]
 \le4k^2\delta_{\mathrm{Blei}}.
\]
Hence some $j_2$ in the support of $p_2$ satisfies
\[
 \mathbb E_{X_1}d_0\bigl(W_1,A_2^2p_2(j_2)\bigr)^2
 \le4k^2\delta_{\mathrm{Blei}}.
\]
Set $z:=A_2^2p_2(j_2)>0$. For every $s$,
\[
 d_0(W_s,z)\le d_0(W_s,W_1)+d_0(W_1,z),
\]
so $(x+y)^2\le2x^2+2y^2$ and \eqref{sharp-eq-pair-metric} give
\begin{equation}
 \mathbb E d_0(W_s,z)^2\le16k^2\delta_{\mathrm{Blei}}
 \qquad(1\le s\le d).
 \label{sharp-eq-common-level}
\end{equation}

Choose the fixed-threshold level sets
\[
 {I_s=\{j\in Y_s:p_s(j)>0,\ d_0(A_s^2p_s(j),z)\le1/2\}
     =\{j\in Y_s:z/3\le A_s^2p_s(j)\le3z\}.}
\]
Write
\[
 m_s=|I_s|,\qquad \theta_s=\sum_{j\in I_s}p_s(j),\qquad
 v_s=\frac{A_s^2}{z},\qquad u_s(j)=v_sp_s(j).
\]
By Markov's inequality and \eqref{sharp-eq-common-level},
\[
 \mathbb P\bigl(d_0(W_s,z)>1/2\bigr)
 \le4\mathbb E d_0(W_s,z)^2\le C_d\delta_{\mathrm{Blei}}.
\]
For small $\delta_{\mathrm{Blei}}$, this probability is less than one, so every $I_s$ is nonempty. Moreover,
\begin{equation}
 1-\theta_s\le C_d\delta_{\mathrm{Blei}},\qquad
 \frac13\le u_s(j)\le3,\qquad
 \sum_{j\in I_s}p_s(j)|u_s(j)-1|^2\le C_d\delta_{\mathrm{Blei}}.
 \label{sharp-eq-level-moments}
\end{equation}
For the last estimate, use
$|u-1|=(u+1)d_0(u,1)\le4d_0(u,1)$ on $[1/3,3]$.
Since $u_s(j)=v_sp_s(j)$,
\[
 \frac{m_s}{v_s}
 =\sum_{j\in I_s}\frac1{v_s}
 =\sum_{j\in I_s}\frac{p_s(j)}{u_s(j)}.
\]
Subtracting $1$ and using $\theta_s=\sum_{j\in I_s}p_s(j)$ gives
\begin{align*}
 \left|\frac{m_s}{v_s}-1\right|
 &\le \left|\sum_{j\in I_s}p_s(j)
       \left(\frac1{u_s(j)}-1\right)\right|+|\theta_s-1|\\
 &\le 3\sum_{j\in I_s}p_s(j)|u_s(j)-1|+C_d\delta_{\mathrm{Blei}}\\
 &\le 3\theta_s^{1/2}
       \left(\sum_{j\in I_s}p_s(j)|u_s(j)-1|^2\right)^{1/2}
       +C_d\delta_{\mathrm{Blei}}\\
 &\le C_d\sqrt\delta_{\mathrm{Blei}}.
\end{align*}
The second line uses $|u^{-1}-1|\le3|u-1|$ on $[1/3,3]$; the third uses Cauchy--Schwarz. Thus
\begin{equation}
 \kappa_s:=\frac{m_s}{v_s}=1+O_d(\sqrt\delta_{\mathrm{Blei}}).
 \label{sharp-eq-cardinality-preliminary}
\end{equation}
In particular, $\kappa_s\in[1/2,2]$ for small $\delta_{\mathrm{Blei}}$.
Now $m_sp_s(j)=\kappa_su_s(j)$, and the function
$x\mapsto1-x^{-1/2}$ is uniformly Lipschitz on $[1/6,6]$.
Equations \eqref{sharp-eq-level-moments} and
\eqref{sharp-eq-cardinality-preliminary} therefore give
\begin{equation}
 \begin{split}
 \sum_{j\in I_s}\left|\sqrt{p_s(j)}-m_s^{-1/2}\right|^2
 &=\sum_{j\in I_s}p_s(j)
     \left|1-(\kappa_su_s(j))^{-1/2}\right|^2\\
 &\le C_d\delta_{\mathrm{Blei}}.
 \end{split}
 \label{sharp-eq-marginal-flatness}
\end{equation}
Let $I={\prod_{s=1}^d}I_s$ and $N_I:={\prod_{s=1}^d}m_s$. Put
$f_s(j)=\sqrt{p_s(j)}$ and $g_s(j)=m_s^{-1/2}$ on $I_s$. Then
\[
 \sqrt{\Pi(i)}=\prod_{s=1}^df_s(i_s),\qquad
 N_I^{-1/2}=\prod_{s=1}^dg_s(i_s),
\]
and
\[
 \prod_{s=1}^df_s-\prod_{s=1}^dg_s
 =\sum_{r=1}^d
   \left({\prod_{s=1}^{r-1}}g_s\right)(f_r-g_r)
   \left({\prod_{s=r+1}^d}f_s\right).
\]
Since counting measure factors over the product and $\|f_s\|_{\ell_2(I_s)}\le1$, $\|g_s\|_{\ell_2(I_s)}=1$, we obtain
\[
 \|\sqrt\Pi-N_I^{-1/2}\|_{\ell_2(I)}
 \le\sum_{r=1}^d\|f_r-g_r\|_{\ell_2(I_r)}.
\]
Using \eqref{sharp-eq-marginal-flatness},
\begin{equation}
 \|\sqrt\Pi-N_I^{-1/2}\|_{\ell_2(I)}\le C_d\sqrt\delta_{\mathrm{Blei}}.
 \label{sharp-eq-product-flatness}
\end{equation}

\medskip
The geometric mean $\sqrt{G/\Pi}=({\prod_{s=1}^d}\eta_s)^{1/k}$ lies between the minimum and maximum of $1,\eta_1,\ldots,\eta_d$. Hence, pointwise,
\[
 |\sqrt{G/\Pi}-1|^2\le{\sum_{s=1}^d}|\eta_s-1|^2.
\]
After multiplication by $\Pi$ and summation,
\begin{equation}
 \|\sqrt G-\sqrt\Pi\|_{\ell_2}^2\le dk^2\delta_{\mathrm{Blei}}.
 \label{sharp-eq-GQ-hellinger}
\end{equation}
Also $\Pi(I^c)\le{\sum_{s=1}^d}(1-\theta_s)$.  Thus
\eqref{sharp-eq-level-moments} and
\eqref{sharp-eq-GQ-hellinger} imply
\begin{equation}
 {\sum_{i\in\Omega\setminus I}}G(i)
 \le2\Pi(I^c)+2\|\sqrt G-\sqrt\Pi\|_{\ell_2}^2
 \le C_d\delta_{\mathrm{Blei}}.
 \label{sharp-eq-box-tail}
\end{equation}
As $B\le1$, this proves
\begin{equation}
 {\sum_{i\in\Omega\setminus I}}a_i^q\le C_d\delta_{\mathrm{Blei}},\qquad
 \|a\|_{\ell_q(I)}\ge1-C_d\varepsilon.
 \label{sharp-eq-box-mass}
\end{equation}

\medskip
On $I$, since $p_1(i_1)=R_1(i_1)/A_1$,
\[
 u_1(i_1)\nu_1(i)
 =\frac{A_1^2p_1(i_1)}{z}
  \frac{a_i^2}{R_1(i_1)A_1}
 =\frac{a_i^2}{z}.
\]
As $a_i\ge0$, this gives
\[
 \frac {a_i}{\sqrt z}=\sqrt{u_1(i_1)}\sqrt{\nu_1(i)}.
\]
Using \eqref{sharp-eq-hellinger},
\eqref{sharp-eq-level-moments}, and
\eqref{sharp-eq-product-flatness}, we obtain
\begin{align*}
 \|a/\sqrt z-N_I^{-1/2}\|_{\ell_2(I)}
 &\le\sqrt3\|\sqrt{\nu_1}-\sqrt\Pi\|_{\ell_2}
   +\|(\sqrt{u_1}-1)\sqrt\Pi\|_{\ell_2(I)}\\
 &\hspace{16mm}+\|\sqrt\Pi-N_I^{-1/2}\|_{\ell_2(I)}
 \le {C_d\sqrt{\delta_{\mathrm{Blei}}}}.
\end{align*}
Indeed, the square of the middle term is at most
$\sum_{j\in I_1}p_1(j)|u_1(j)-1|^2$.
Put
\[
 \rho_I:=\sqrt{z/N_I},\qquad b:=\frac1{N_I}\sum_{i\in I}a_i.
\]
{Consequently,}
\begin{equation}
 \begin{aligned}
 \frac1{N_I}\sum_{i\in I}|a_i-\rho_I|^2&\le C_d\delta_{\mathrm{Blei}}\rho_I^2,\\
 b&=\rho_I(1+O_d(\sqrt\delta_{\mathrm{Blei}})),\\
 \frac1{N_I}\sum_{i\in I}|a_i/b-1|^2&\le C_d\delta_{\mathrm{Blei}}.
 \end{aligned}
 \label{sharp-eq-box-variance}
\end{equation}

\medskip
\noindent \emph{Step 3: Balancing the side lengths.}
Apply Lemma~\ref{lem:q-moment-variance} on $I$ with normalized
counting measure and $x_i=a_i/b$.  By definition of $b$,
\[
 \frac1{N_I}\sum_{i\in I}x_i=1,
\]
and \eqref{sharp-eq-box-variance} gives
\[
 \frac1{N_I}\sum_{i\in I}(x_i-1)^2\le C_d\delta_{\mathrm{Blei}}.
\]
Hence
\[
 1\le \frac1{N_I}\sum_{i\in I}x_i^q\le1+C_d\delta_{\mathrm{Blei}}.
\]
Multiplying by $b^qN_I$ and taking the $q$th root gives
\begin{equation}
 \|a\|_{\ell_q(I)}
 =bN_I^{1/q}\left(\frac1{N_I}\sum_{i\in I}x_i^q\right)^{1/q}
 =bN_I^{1/q}(1+O_d(\delta_{\mathrm{Blei}})).
 \label{sharp-eq-mean-q-comparison}
\end{equation}
For each slot $s$,
\[
 N_Ib=\sum_{i\in I}a_i
 =\sum_{j\in I_s}\sum_{\substack{i\in I\\ i_s=j}}a_i.
\]
Each selected slice contains $N_I/m_s$ indices, so Cauchy--Schwarz gives
\[
 \sum_{\substack{i\in I\\ i_s=j}}a_i
 \le\sqrt{\frac{N_I}{m_s}}
      \left(\sum_{\substack{i\in I\\ i_s=j}}a_i^2\right)^{1/2}
 \le\sqrt{\frac{N_I}{m_s}}R_s(j).
\]
Summing over $j\in I_s$ yields
\[
 N_Ib\le\sqrt{\frac{N_I}{m_s}}A_s,
\]
and therefore
\[
 1\ge A_s\ge b\sqrt{N_I m_s}
 =bN_I^{1/q}\sqrt{m_s/N_I^{1/d}}.
\]
Hence \eqref{sharp-eq-box-mass} and
\eqref{sharp-eq-mean-q-comparison} imply
\begin{equation}
 \max_s m_s\le(1+C_d\varepsilon)N_I^{1/d}.
 \label{sharp-eq-side-balance}
\end{equation}
Put
\[
 r_s:=\frac{m_s}{N_I^{1/d}}.
\]
Then $\prod_{s=1}^d r_s=1$, while
\eqref{sharp-eq-side-balance} gives $r_s\le1+C_d\varepsilon$ for every
$s$.  If $r_{s_0}=\min_s r_s$, then
\[
 1=\prod_{s=1}^d r_s
 \le r_{s_0}(1+C_d\varepsilon)^{d-1},
\]
so
\[
 r_{s_0}\ge(1+C_d\varepsilon)^{-(d-1)}
          \ge1-C_d\varepsilon
\]
for small $\varepsilon$. Thus every $r_s$ lies between $1-C_d\varepsilon$ and $1+C_d\varepsilon$, giving
\[
 \frac{\max_s m_s}{\min_s m_s}\le1+C_d\varepsilon
\]
after enlarging the constant.

\noindent \emph{Step 4: Trimming and normalization.}
Set $m=\min_sm_s$ and $N=m^d$.  Select subsets
$J_s\subseteq I_s$ of cardinality $m$ so that
\begin{equation}
 \sum_{i\in J}a_i^q\ge\frac N{N_I}\sum_{i\in I}a_i^q,
 \qquad J={\prod_{s=1}^d}J_s.
 \label{sharp-eq-random-trimming}
\end{equation}
Choose each $J_s$ independently and uniformly among the $m$-element
subsets of $I_s$. Then
\begin{align*}
 \mathbb E\sum_{i\in J}a_i^q
 &=\sum_{i\in I}a_i^q
   \mathbb P(i_1\in J_1,\ldots,i_d\in J_d)\\
 &=\sum_{i\in I}a_i^q\prod_{s=1}^d\frac{m}{m_s}
 =\frac{N}{N_I}\sum_{i\in I}a_i^q.
\end{align*}
Some choice of the $J_s$ therefore satisfies \eqref{sharp-eq-random-trimming}. Moreover, \eqref{sharp-eq-side-balance} gives $N/N_I\ge1-C_d\varepsilon$.
Thus \eqref{sharp-eq-box-mass} and
\eqref{sharp-eq-random-trimming} imply
\eqref{sharp-eq-tail} and
\begin{equation}
 1-C_d\varepsilon\le\|a\|_{\ell_q(J)}\le1.
 \label{sharp-eq-final-cube-mass}
\end{equation}

Restricting the first estimate in \eqref{sharp-eq-box-variance} to
$J$, and using $N_I/N=1+O_d(\varepsilon)$, gives
\[
 \frac1N\sum_{i\in J}|a_i-\rho_I|^2\le C_d\varepsilon\rho_I^2.
\]
If $b_J:=N^{-1}\sum_{i\in J}a_i$, then Cauchy--Schwarz and \eqref{sharp-eq-box-variance} give
\[
 |b_J-\rho_I|
 \le\left(\frac1N\sum_{i\in J}|a_i-\rho_I|^2\right)^{1/2}
 \le C_d\sqrt\varepsilon\,\rho_I.
\]
Consequently $b_J/\rho_I=1+O_d(\sqrt\varepsilon)$ and
\[
 \frac1N\sum_{i\in J}|a_i/b_J-1|^2\le C_d\varepsilon.
\]
Applying Lemma~\ref{lem:q-moment-variance} again, now on $J$ with $x_i=a_i/b_J$, yields
\[
 \|a\|_{\ell_q(J)}=b_JN^{1/q}(1+O_d(\varepsilon)).
\]
{Equation \eqref{sharp-eq-final-cube-mass} gives}
\begin{equation}
 b_JN^{1/q}=1+O_d(\varepsilon).
 \label{sharp-eq-final-mean}
\end{equation}
Set
\[
 \rho:=N^{-1/q}=m^{-(d+1)/2}.
\]
{By \eqref{sharp-eq-final-mean},}
\[
 |b_J-\rho|\le C_d\varepsilon\rho.
\]
Together with $b_J/\rho_I=1+O_d(\sqrt\varepsilon)$, this gives
\[
 \frac{\rho_I}{\rho}=1+O_d(\sqrt\varepsilon).
\]
Moreover,
\begin{align*}
 \frac1N\sum_{i\in J}|a_i-\rho|^2
 &\le \frac2N\sum_{i\in J}|a_i-b_J|^2+2|b_J-\rho|^2\\
 &\le C_d\varepsilon\rho_I^2+C_d\varepsilon^2\rho^2\\
 &\le C_d\varepsilon\rho^2.
\end{align*}
{Dividing by $\rho^2$ gives \eqref{sharp-eq-variance}.}  Since
\[
 \frac1N\sum_{i\in J}a_i^2
 =\frac1N\sum_{i\in J}|a_i-b_J|^2+b_J^2,
\]
\eqref{sharp-eq-variance} and \eqref{sharp-eq-final-mean} give both estimates in
\eqref{sharp-eq-moments}.

For a selected coordinate $j\in J_s$, the definition of $I_s$ gives
\[
 R_s(j)=A_sp_s(j)\le\frac{3z}{A_s}.
\]
By \eqref{sharp-eq-cardinality-preliminary},
\[
 \frac{z}{A_s^2}=\frac{1+O_d(\sqrt\delta_{\mathrm{Blei}})}{m_s}.
\]
Consequently
\[
 R_s(j)\le \frac{3z}{A_s}
 =\frac{3A_s}{m_s}\bigl(1+O_d(\sqrt\delta_{\mathrm{Blei}})\bigr)
 \le \frac{C_d}{m_s}
 \le \frac{C_d}{m},
\]
where we used $A_s\le1$ and $m_s\ge m$. This is
\eqref{sharp-eq-full-slices}.
Finally, probability $L^q$ norms are bounded by their $L^2$ counterparts because $q<2$. Together with $\rho N^{1/q}=1$, \eqref{sharp-eq-variance} bounds the interior coefficient error by $C_d\sqrt\varepsilon$; \eqref{sharp-eq-tail} bounds the exterior error by $C_d\varepsilon^{1/q}\le C_d\sqrt\varepsilon$. This proves \eqref{sharp-eq-flat-approximation}.
{For $\varepsilon=0$, all error terms vanish.}
\end{proof}

\section{Compatible paths and edge extraction}\label{sec:compatible-paths}

Throughout this section, finite sets carry normalized counting measure.
For $z\in\mathbb K$, set
\[
 \operatorname{ph}(z):=\begin{cases}z/|z|,&z\ne0,\\1,&z=0.\end{cases}
\]
Thus $|\operatorname{ph}(z)|=1$ for every $z$.
A matrix indexed by $A$ and $B$ acts from $L^2(B)$ to $L^2(A)$ by
$(Mv)(a)=|B|^{-1}\sum_{b\in B}M(a,b)v(b)$; its squared
Hilbert--Schmidt norm is $\mathbb E_{a,b}|M(a,b)|^2$.

\begin{lemma}
\label{lem:anchored-compatible-path}
Let $d\geq2$, let $X_1,\ldots,X_d$ be nonempty finite sets, and let
$F:X_1\times\cdots\times X_d\longrightarrow\mathbb K$, where
$\mathbb K\in\{\mathbb R,\mathbb C\}$, satisfy $|F|=1$.
For $2\leq s\leq d-1$ and $j\in X_s$, let $\mathsf M_{s,j}(F)$ be the
conditional matrix of $F$, with row variable $(x_1,\ldots,x_{s-1})$
and column variable $(x_{s+1},\ldots,x_d)$, and put
{\begin{equation}\label{eq:conditional-rank-defect}
 \Delta_s(F)^2=\frac1{|X_s|}\sum_{j\in X_s}
 \operatorname{dist}_{\mathrm{HS}}
 (\mathsf M_{s,j}(F),\{M:\operatorname{rank}M\le1\})^2,
 \qquad \Delta_s=\Delta_s(F).
\end{equation}}
There are unimodular functions $H_r:X_r\times X_{r+1}\to\mathbb K$
such that
\[
 \left\|F-\prod_{r=1}^{d-1}H_r(x_r,x_{r+1})\right\|_{L^2}
 \leq8\sum_{s=2}^{d-1}\Delta_s
 \leq8\sqrt{(d-2)\sum_{s=2}^{d-1}\Delta_s^2}.
\]
The same estimate holds in $L^q$ for $1\leq q\leq2$. For $d=2$
one can take $H_1=F$, with zero error.
\end{lemma}

\begin{proof}
If $d=2$, take $H_1=F$; the estimate is then an identity. Assume $d\ge3$. For each $s,j$, put
\[
 p=(x_1,\ldots,x_{s-1}),\qquad y=(x_{s+1},\ldots,x_d),
\]
choose a best rank-one approximation, and write it as
\[
 W_{s,j}(p,y)=u_{s,j}(p)v_{s,j}(y).
\]
Define
\[
 \Phi_s(p,j):=\operatorname{ph}(u_{s,j}(p)),\qquad
 \Psi_s(j,y):=\operatorname{ph}(v_{s,j}(y)),
\]
and
\[
 \Theta_s(p,j,y):=\Phi_s(p,j)\Psi_s(j,y).
\]
Then $|\Theta_s|=1$ and every conditional matrix of $\Theta_s$ at the $s$th coordinate has rank one.
For $|z|=1$ and $w\neq0$,
\[
 |z-w/|w||\leq |z-w|+\big||w|-1\big|\leq2|z-w|.
\]
If $W_{s,j}(p,y)=0$, then
$|F(p,j,y)-\Theta_s(p,j,y)|\le2=2|F(p,j,y)-W_{s,j}(p,y)|$. Thus
\begin{equation}\label{eq:phase-rounded-slice-error}
 \|F-\Theta_s\|_{L^2}\leq2\Delta_s.
\end{equation}
In the real case these phase choices are simply signs.

Let $a=(a_1,\ldots,a_d)$ be uniformly distributed on $X_1\times\cdots\times X_d$. For
$1\leq r\leq d-2$, define
\[
 H_r^a(x_r,x_{r+1})=
 \frac{F(a_1,\ldots,a_{r-1},x_r,x_{r+1},a_{r+2},\ldots,a_d)}
      {F(a_1,\ldots,a_r,x_{r+1},a_{r+2},\ldots,a_d)},
\]
and set
\[
 H_{d-1}^a(x_{d-1},x_d)
 =F(a_1,\ldots,a_{d-2},x_{d-1},x_d).
\]
All denominators have modulus one. Write
\[
 F_r^a(x)=F(a_1,\ldots,a_r,x_{r+1},\ldots,x_d),
 \qquad F_0^a=F,
\]
{and $E_r^a=F_{r-1}^a-H_r^aF_r^a$ for $1\le r\le d-2$.} Multiplication by the
unimodular denominator in the definition of $H_r^a$ expresses
$E_r^a$ as
\begin{equation}\label{eq:anchored-minor}
 F(p,j,y)F(p',j,b)-F(p,j,b)F(p',j,y),
\end{equation}
where
\[
 \begin{split}
 p&=(a_1,\ldots,a_{r-1},x_r),\qquad
 p'=(a_1,\ldots,a_r),\qquad j=x_{r+1},\\
 y&=(x_{r+2},\ldots,x_d),\qquad
 b=(a_{r+2},\ldots,a_d).
 \end{split}
\]
For the conditional rank-one function
\[
 \Theta_{r+1}(p,j,y)=\Phi_{r+1}(p,j)\Psi_{r+1}(j,y),
\]
the corresponding minor is identically zero, because
\begin{align*}
 &\Theta_{r+1}(p,j,y)\Theta_{r+1}(p',j,b)
 -\Theta_{r+1}(p,j,b)\Theta_{r+1}(p',j,y)\\
 &=\Phi_{r+1}(p,j)\Phi_{r+1}(p',j)
   \bigl[\Psi_{r+1}(j,y)\Psi_{r+1}(j,b)
        -\Psi_{r+1}(j,b)\Psi_{r+1}(j,y)\bigr]\\
 &=0.
\end{align*}
Subtracting this identity from \eqref{eq:anchored-minor} and replacing
one factor at a time gives
\begin{align*}
 |E_r^a|
 &\le |F(p,j,y)-\Theta_{r+1}(p,j,y)|
      +|F(p',j,b)-\Theta_{r+1}(p',j,b)|\\
 &\quad+|F(p,j,b)-\Theta_{r+1}(p,j,b)|
      +|F(p',j,y)-\Theta_{r+1}(p',j,y)|,
\end{align*}
{because all values of $F$ and $\Theta_{r+1}$ have modulus one. Each evaluation
point has the uniform product distribution in $(a,x)$. Minkowski's inequality
and \eqref{eq:phase-rounded-slice-error} give}
\[
 \|E_r^a(x)\|_{L^2(a,x)}
 \leq4\|F-\Theta_{r+1}\|_{L^2}\leq8\Delta_{r+1}.
\]
{The telescoping identity}
\[
 F-\prod_{r=1}^{d-1}H_r^a
 =\sum_{r=1}^{d-2}\left({\prod_{k=1}^{r-1}}H_k^a\right)E_r^a
\]
and $|H_k^a|=1$ bound the left side in $L^2(a,x)$ by
$8{\sum_{s=2}^{d-1}} \Delta_s$. Some anchor attains this bound.
Since the measure is normalized, the $L^q$ estimate follows.
\end{proof}

\begin{lemma}
\label{lem:local-path-perturbation}
Let $F$ be arbitrary on $X_1\times\cdots\times X_d$, and define
$\Delta_s(F)$ as in Lemma~\ref{lem:anchored-compatible-path}. Put
$b=\||F|-1\|_{L^2}$. Then a unimodular path $P$ exists with
\[
 \|F-P\|_{L^2}
 \leq\bigl(1+8(d-2)\bigr)b+8\sum_{s=2}^{d-1}\Delta_s(F).
\]

\end{lemma}

\begin{proof}
Define $Z(x):=\operatorname{ph}(F(x))$. Then $|Z|=1$ and
$|F(x)-Z(x)|=\bigl||F(x)|-1\bigr|$ for every $x$, so $\|F-Z\|_{L^2}=b$. For $2\le s\le d-1$ and $j\in X_s$,
write $\mathcal R_1$ for the rank-one cone in the conditional matrix
space. Distance to this cone is $1$-Lipschitz, so
\[
 \operatorname{dist}_{\mathrm{HS}}(\mathsf M_{s,j}(Z),\mathcal R_1)
 \le \operatorname{dist}_{\mathrm{HS}}(\mathsf M_{s,j}(F),\mathcal R_1)
    +\|\mathsf M_{s,j}(Z-F)\|_{\mathrm{HS}}.
\]
Taking the $L^2$ norm in the fixed coordinate $j$ and applying
Minkowski's inequality gives
\[
 \Delta_s(Z)\le \Delta_s(F)
 +\left(\mathbb E_j\|\mathsf M_{s,j}(Z-F)\|_{\mathrm{HS}}^2\right)^{1/2}
 =\Delta_s(F)+\|Z-F\|_{L^2}
 =\Delta_s(F)+b.
\]
Lemma~\ref{lem:anchored-compatible-path} and the triangle inequality
prove the assertion. The case $d=2$
follows by taking $P=Z$.
\end{proof}

{For a matrix $A\in\mathbb K^{m\times m}$, define
\begin{equation}\label{eq:beta-definition}
 \beta(A):=\sup\left\{
 \left|\sum_{a,b=1}^m A(a,b)\langle u_a,v_b\rangle\right|:
 \|u_a\|\le1,\ \|v_b\|\le1\ (a,b\in[m])\right\}.
\end{equation}
The test vectors lie in a common finite-dimensional Hilbert space over
$\mathbb K$. Since their span has dimension at most $2m$, we may use
$\mathbb K^{2m}$.}

\begin{lemma}
\label{lem:robust-edge-extraction}
Let $J_1,\ldots,J_d$ be coordinate subsets of equal cardinality $m$,
identified separately with $[m]$, and let $\rho>0$. Write
$F=T_J/\rho$ for the restriction of the coefficient tensor $T$ to
$J=J_1\times\cdots\times J_d$, divided by $\rho$. Suppose
$P(x)=\prod_{s=1}^{d-1}H_s(x_s,x_{s+1})$ is a unimodular path.
For $1\leq r\leq d-1$ define
{\begin{equation}\label{eq:edge-average-definition}
 \mathsf A_r(a,b)=\frac1{m^{d-2}}
 \sum_{\substack{x\in[m]^d\\x_r=a,\ x_{r+1}=b}}
 F(x)\overline{\prod_{\substack{1\le s\le d-1\\s\ne r}}
 H_s(x_s,x_{s+1})}.
\end{equation}
This is the uniform average over the $d-2$ coordinates other than
$x_r$ and $x_{r+1}$.}
Then
\begin{equation}\label{eq:robust-edge-bound}
 \|T\|_{\mathrm{cb}}
 \geq\rho m^{(d-2)/2}\beta(\mathsf A_r).
\end{equation}
Moreover, for every $1\leq q\leq\infty$,
\begin{equation}\label{eq:averaged-edge-distance}
 \|\mathsf A_r-H_r\|_{L^q([m]^2)}\leq\|F-P\|_{L^q([m]^d)}.
\end{equation}
In particular, if $\|T\|_{\mathrm{cb}}\leq
\rho m^{(d+1)/2}(1+\delta)$, then
$\beta(\mathsf A_r)\leq(1+\delta)m^{3/2}$.
\end{lemma}

\begin{proof}
The test below isolates the \(r\)th edge of the path: the blocks on the two
sides transport the unimodular edge phases, while the block at level \(r\)
inserts the bilinear pairing \(\langle u_a,v_b\rangle\). After summation over
the remaining coordinates, the resulting completely bounded test is exactly
the pairing with \(\mathsf A_r\).
Fix vectors $u_a,v_b$ in a finite-dimensional Hilbert space $\mathcal V$,
all of norm at most one. Set $\mathcal V_0=\mathcal V_d=\mathbb K$, $\mathcal V_r=\mathcal V$,
and $\mathcal V_s=\ell_2^m$ for the remaining internal levels. We construct
contractions $V_i^s:\mathcal V_s\to\mathcal V_{s-1}$.

If $r<d-1$, on the right of the selected edge put
\[
 V_i^d(1)=e_i,
 \qquad
 V_i^s e_j=m^{-1/2}\overline{H_s(i,j)}e_i
 \quad(r+2\leq s\leq d-1),
\]
and
\[
 V_b^{r+1}e_j=m^{-1/2}\overline{H_{r+1}(b,j)}v_b.
\]
If $r=d-1$, replace this entire right-hand construction by
$V_b^d(1)=v_b$.

If $r>1$, on the left put
\[
 V_a^r z=\langle u_a,z\rangle m^{-1/2}
       \sum_{j=1}^m\overline{H_{r-1}(j,a)}e_j,
\]
\[
 V_i^s e_j=\mathbf 1_{\{j=i\}}m^{-1/2}
       \sum_{k=1}^m\overline{H_{s-1}(k,i)}e_k
 \quad(2\leq s\leq r-1),
 \qquad
 V_i^1e_j=\mathbf 1_{\{j=i\}}.
\]
If $r=1$, the left side consists only of
$V_a^1z=\langle u_a,z\rangle$. Each displayed map is a contraction.  For example, if
$r+2\le s\le d-1$, then for $z=\sum_{j=1}^mz_je_j$,
\[
 V_i^sz=m^{-1/2}\sum_{j=1}^mz_j\overline{H_s(i,j)}e_i,
\]
so Cauchy--Schwarz and $|H_s(i,j)|=1$ give
\begin{align*}
 \|V_i^sz\|^2
 &=\frac1m\left|\sum_{j=1}^mz_j\overline{H_s(i,j)}\right|^2\\
 &\le\frac1m\left(\sum_{j=1}^m|z_j|^2\right)
                 \left(\sum_{j=1}^m|H_s(i,j)|^2\right)\\
 &=\sum_{j=1}^m|z_j|^2=\|z\|^2.
\end{align*}
{If $r<d-1$, the same calculation for the map at level $r+1$ gives}
\[
 \|V_b^{r+1}z\|
 \le m^{-1/2}\left({\sum_{j=1}^m|z_j|^2}\right)^{1/2}
                    \left({\sum_{j=1}^m|H_{r+1}(b,j)|^2}\right)^{1/2}
                    \|v_b\|
 \le\|z\|.
\]
{If $r>1$, the vector on the left}
\[
 m^{-1/2}\sum_{j=1}^m\overline{H_{r-1}(j,a)}e_j
\]
has norm one. Therefore
\[
 \|V_a^rz\|=|\langle u_a,z\rangle|
 \le\|u_a\|\,\|z\|\le\|z\|.
\]
For $2\le s\le r-1$,
\[
 \|V_i^sz\|=|z_i|\le\|z\|,
\]
and $V_i^1$ is a coordinate functional.  The endpoint definitions
$r=1$ and $r=d-1$ are contractions because $\|u_a\|,\|v_b\|\le1$.

Put $\mathcal W:=\bigoplus_{s=0}^d\mathcal V_s$. For each $s\in[d]$ and $i\in J_s$, extend $V_i^s:\mathcal V_s\to\mathcal V_{s-1}$ to an operator on $\mathcal W$ by setting it equal to zero on $\mathcal V_s^\perp$; for $i\notin J_s$, set $V_i^s=0$ on $\mathcal W$. These extensions have the same operator norms as the displayed maps and hence are contractions. Thus $(V_i^s)_{s\in[d],\,i\in[n]}$ is an admissible family in \eqref{eq:cb-def-intro}.

{If $r<d-1$, then}
\[
 V_{x_{r+2}}^{r+2}\cdots V_{x_d}^d(1)
 =m^{-(d-r-2)/2}
 \overline{\prod_{s=r+2}^{d-1}H_s(x_s,x_{s+1})}\,e_{x_{r+2}},
\]
and therefore
\[
 V_{x_{r+1}}^{r+1}\cdots V_{x_d}^d(1)
 =m^{-(d-r-1)/2}
 \overline{\prod_{s=r+1}^{d-1}H_s(x_s,x_{s+1})}\,v_{x_{r+1}}.
\]
For $r=d-1$ the same formula reads simply
$V_{x_d}^d(1)=v_{x_d}$. Applying the selected-edge map gives, when
$r>1$,
\[
 V_{x_r}^r\cdots V_{x_d}^d(1)
 =m^{-(d-r)/2}
 \overline{\prod_{s=r+1}^{d-1}H_s(x_s,x_{s+1})}
 \langle u_{x_r},v_{x_{r+1}}\rangle
 \sum_{j=1}^m\overline{H_{r-1}(j,x_r)}e_j.
\]
For $2\le s\le r-1$, the definition of the left-hand block gives
\[
 V_{x_s}^s e_j
 =\mathbf 1_{\{j=x_s\}}m^{-1/2}
   \sum_{k=1}^m\overline{H_{s-1}(k,x_s)}e_k,
\]
while $V_{x_1}^1e_j=\mathbf 1_{\{j=x_1\}}$.
{For $r\ge3$, applying $V_{x_{r-1}}^{r-1}$ to the preceding sum gives
\begin{align*}
 V_{x_{r-1}}^{r-1}V_{x_r}^r\cdots V_{x_d}^d(1)
 &=m^{-(d-r+1)/2}
   \overline{H_{r-1}(x_{r-1},x_r)
   \prod_{s=r+1}^{d-1}H_s(x_s,x_{s+1})}\\
 &\quad{}\times\langle u_{x_r},v_{x_{r+1}}\rangle
   \sum_{k=1}^m\overline{H_{r-2}(k,x_{r-1})}e_k.
\end{align*}
Iterating the same identity down to level $1$ yields}
\[
 V_{x_1}^1\cdots V_{x_d}^d(1)
 =m^{-(d-2)/2}
   \overline{{\prod_{\substack{1\le s\le d-1\\s\ne r}}}H_s(x_s,x_{s+1})}
   \langle u_{x_r},v_{x_{r+1}}\rangle.
\]
{For $r=2$, the same scalar identity follows by applying the
coordinate functional $V_{x_1}^1$ directly to
$V_{x_2}^2\cdots V_{x_d}^d(1)$.}
For $r=1$ the left selection is absent and the same formula follows
directly from $V_{x_1}^1z=\langle u_{x_1},z\rangle$.
{For $d=2$, the product over $s\ne r$ is empty and equals one.}
Coefficients outside $J$ contribute zero. Multiplying by the
coefficients of $T$ and summing over the other $d-2$ coordinates gives
\[
 \rho m^{-(d-2)/2}m^{d-2}
 \sum_{a,b=1}^m \mathsf A_r(a,b)\langle u_a,v_b\rangle.
\]
Its absolute value is bounded by $\|T\|_{\mathrm{cb}}$. Taking
the supremum proves \eqref{eq:robust-edge-bound}.

Finally,
\[
 \mathsf A_r-H_r=\mathbb E\left[
 (F-P)\overline{{\prod_{\substack{1\le s\le d-1\\s\ne r}}}H_s}
 \;\middle|\;x_r,x_{r+1}\right].
\]
The phases preserve absolute values, and conditional expectation
contracts $L^q$. This proves \eqref{eq:averaged-edge-distance}.
\end{proof}

\section{Nuclear and polar estimates}

\subsection{The bilinear nuclear estimate}

{For an $m\times m$ matrix $A$, the normalized probability-space norms and
counting-measure norms are related by
\[
\|A\|_{L^{4/3}([m]^2)}=m^{-3/2}\|A\|_{\ell_{4/3}},
\qquad
\|A\|_{\mathrm{HS},\mathrm{prob}}=m^{-1}\|A\|_F.
\]
}
\begin{lemma}\label{lem:matrix-interpolation}
For every $B\in\mathbb K^{m\times m}$,
\[
 \left({\sum_{i,j=1}^m}|B_{ij}|^{4/3}\right)^{3/2}
       \le \beta(B)\|B\|_{S_1}.
\]
\end{lemma}
\begin{proof}
We obtain the estimate by expressing \(\beta(B)\) through a Gram-matrix
optimization. Its dual formulation provides diagonal weights for \(B\), which
can then be combined with Schatten duality and Hölder's inequality.
{Write $G\succeq0$ when the Hermitian matrix $G$ is positive
semidefinite, that is, $z^*Gz\ge0$ for every column vector $z$.
Conjugating both families of vectors in \eqref{eq:beta-definition} gives
$\beta(B)=\beta(\overline B)$. Define
\[
 S_B=\frac12\begin{pmatrix}0&B\\B^*&0\end{pmatrix}.
\]
Let $w=(u_1,\ldots,u_m,v_1,\ldots,v_m)$, and let
$G_{\alpha\beta}=\langle w_\alpha,w_\beta\rangle$ be its Gram matrix. Then
\[
 \Tr(S_BG)=\operatorname{Re}\sum_{i,j=1}^m
 \overline{B_{ij}}\langle u_i,v_j\rangle.
\]
A common phase on one vector family makes the real-part and
absolute-value maxima equal. Thus the Gram formulation is
\[
 \beta(B)=\max\{\Tr(S_BG):G\succeq0,\ G_{\alpha\alpha}\le1\ (1\le\alpha\le2m)\}.
\]
Its semidefinite dual, after multiplying the constraint by two and
conjugating by $\operatorname{diag}(I_m,-I_m)$, is
\begin{equation}\label{eq:beta-dual}
 \begin{split}
 \beta(B)=\inf\biggl\{\frac12\biggl(\sum_{i=1}^m a_i+\sum_{j=1}^m b_j\biggr):\ &
 a_i,b_j\ge0,\\[-2pt]
 &\begin{pmatrix}D_a&B\\B^*&D_b\end{pmatrix}\succeq0\biggr\},
 \end{split}
\end{equation}
where $D_a=\operatorname{diag}(a_1,\ldots,a_m)$ and
$D_b=\operatorname{diag}(b_1,\ldots,b_m)$.
Strict feasibility holds for the primal at $G=\tfrac12 I_{2m}$ and
for the dual at $a_i=b_j=t>\|B\|_{\mathrm{op}}$.
Strong duality applies over both fields.}

First consider a feasible pair with every $a_i,b_j>0$. The Schur
complement criterion shows that
$K=D_a^{-1/2}BD_b^{-1/2}$ is a contraction. Schatten duality yields
\[
 \|B\|_{S_1}\ge \operatorname{Re}\operatorname{Tr}(K^*B)
       ={\sum_{i,j=1}^m}\frac{|B_{ij}|^2}{\sqrt{a_i b_j}}.
\]
H\"older's inequality with exponents $3/2$ and $3$ gives
\[
 {\sum_{i,j=1}^m}|B_{ij}|^{4/3}
 \le\left({\sum_{i,j=1}^m}\frac{|B_{ij}|^2}{\sqrt{a_i b_j}}\right)^{2/3}
       \left({\sum_{i,j=1}^m}a_i b_j\right)^{1/3}.
\]
Raising to the power $3/2$ and applying the arithmetic--geometric mean
inequality yields
\[
 \left({\sum_{i,j=1}^m}|B_{ij}|^{4/3}\right)^{3/2}
 \le\|B\|_{S_1}\sqrt{\Bigl({\sum_{i=1}^m} a_i\Bigr)\Bigl({\sum_{j=1}^m} b_j\Bigr)}
 \le\frac12\|B\|_{S_1}\left({\sum_{i=1}^m} a_i+{\sum_{j=1}^m} b_j\right).
\]
For a feasible pair with zero entries, $a_i+t,b_j+t$ remains feasible
and is strictly positive. Let $t\downarrow0$ and take the dual infimum
to conclude.
\end{proof}

\subsection{\texorpdfstring{{Approximation by unitary matrices}}{Approximation by unitary matrices}}
\begin{lemma}\label{lem:edge-replacement}
Let $J_1,\ldots,J_d$ have common cardinality $m$.  Let
$K_t:J_t\times J_{t+1}\to\mathbb K$ satisfy
$K_tK_t^*=mI_m$, and let $H_t:J_t\times J_{t+1}\to\mathbb K$ be
unimodular. {Identify each $J_t$ with $[m]$.} Fix $1\le r\le d-1$ and put
\[
 D_r(i_r,i_{r+1})=H_r(i_r,i_{r+1})-K_r(i_r,i_{r+1}).
\]
Then, with normalized counting measure on $J=J_1\times\cdots\times J_d$,
\[
 \left\|
 \left({\prod_{t=1}^{r-1}}K_t(i_t,i_{t+1})\right)
 D_r(i_r,i_{r+1})
 \left({\prod_{t=r+1}^{d-1}}H_t(i_t,i_{t+1})\right)
 \right\|_{L^2(J)}
 =\frac{\|D_r\|_F}{m}.
\]
\end{lemma}

\begin{proof}
The unimodular factors to the right of $D_r$ disappear under
absolute values. The squared $L^2$ norm is therefore
\[
 \frac1{m^d}{\sum_{i_1,\ldots,i_d=1}^m}
 \left|{\prod_{t=1}^{r-1}}K_t(i_t,i_{t+1})\right|^2
 |D_r(i_r,i_{r+1})|^2.
\]
The variables $i_{r+2},\ldots,i_d$ do not occur in the summand and
contribute the factor $m^{d-r-1}$.  For fixed $i_r$, summing first in
$i_1$ and using {$K_1^*K_1=mI_m$, which follows from $K_1K_1^*=mI_m$ because $K_1$ is square,} gives
\[
 {\sum_{i_1=1}^m}|K_1(i_1,i_2)|^2=m.
\]
More generally, for $1\le \ell\le r-1$, summing in
$i_1,\ldots,i_\ell$ gives
\[
 {\sum_{i_1,\ldots,i_\ell=1}^m}
 \prod_{t=1}^{\ell}|K_t(i_t,i_{t+1})|^2=m^\ell,
\]
which follows inductively from
\[
 \sum_{i_\ell=1}^m|K_\ell(i_\ell,i_{\ell+1})|^2=m.
\]
Taking $\ell=r-1$ gives the factor $m^{r-1}$. Therefore the squared norm is
\begin{align*}
 \frac{m^{d-r-1}m^{r-1}}{m^d}
 {\sum_{i_r,i_{r+1}=1}^m}|D_r(i_r,i_{r+1})|^2
 &=\frac1{m^2}\|D_r\|_F^2.
\end{align*}
Taking square roots proves the identity. At $r=1$ or $r=d-1$, the corresponding sequence of summations is absent, and its factor is $m^0=1$.
\end{proof}

\begin{lemma}
\label{lem:sqrt-spectral-polar-closure}
Fix $d\ge2$, put $q=2d/(d+1)$, and suppose that
$\|T\|_{\mathrm{cb}}=1$ and $\|T\|_q\ge1-\varepsilon$,
{where $0\le\varepsilon\le\varepsilon_d$ and $\varepsilon_d>0$
is sufficiently small, depending only on $d$.}
Let $J=\prod_{s=1}^dJ_s$, where $|J_s|=m$, and define
\[
 \rho=m^{-(d+1)/2},\qquad F=\rho^{-1}T|_J.
\]
All expectations on $J$ and its coordinate products use uniform
probability measure. Suppose, with constants depending only on $d$,
\begin{align}
 \|T-T\mathbf1_J\|_q&\le C_d\varepsilon^{1/q},
 &\mathbb E(|F|-1)^2&\le C_d\varepsilon,\label{eq:sqrt-cube-input-one}\\
 |\mathbb E|F|^2-1|&\le C_d\varepsilon,
 &\max_{j\in J_s}h_{s,j}&\le C_d/m
 \quad(1\le s\le d),\label{eq:sqrt-cube-input-two}
\end{align}
where $h_{s,j}$ is the Frobenius norm of the full counting-measure slice of $T$ at $i_s=j$.
Then there exist unimodular edge matrices $H_r$ and scaled unitaries
$K_r=\sqrt m\,U_r$, $1\le r\le d-1$, such that the tensors
\[
 P(e_{i_1},\ldots,e_{i_d})=\rho\prod_{r=1}^{d-1}H_r(i_r,i_{r+1}),\qquad
 Q(e_{i_1},\ldots,e_{i_d})=\rho\prod_{r=1}^{d-1}K_r(i_r,i_{r+1})
\]
on $J$, extended by zero, satisfy
\begin{align}
 \|T-P\|_q+\|T-Q\|_q&\le C_d\sqrt\varepsilon,\label{eq:sqrt-chain-conclusion}\\
 \max_{1\le r\le d-1}\frac{\|H_r-K_r\|_F}{m}&\le C_d\sqrt\varepsilon,\qquad
 \|P\|_q=\|Q\|_{\mathrm{cb}}=1.\label{eq:sqrt-edge-conclusion}
\end{align}
The assertions hold over both fields.
\end{lemma}

\begin{proof}
For the full conditional slices write
\[
 A_s={\sum_{j=1}^n h_{s,j}},\qquad
 {\Gamma_s}={\sum_{\substack{1\le j\le n\\h_{s,j}>0}}}\frac{h_{s,j}^2}{\|M_{s,j}\|_{\mathrm{op}}}.
\]
{Lemmas~\ref{lem:spectral-input} and \ref{lem:blei} give}
\[
 A_s\le {\Gamma_s}\le1,\qquad
 \prod_{s=1}^d A_s\ge\|T\|_q^d\ge(1-\varepsilon)^d.
\]
Each factor is at most one, so $A_s\ge(1-\varepsilon)^d$ for every
$s$. Hence
\begin{equation}\label{eq:sqrt-spectral-gap}
 {\Gamma_s}-A_s\le d\varepsilon.
\end{equation}
For a nonzero matrix with Frobenius norm $h$ and largest singular value
$\sigma$, its squared distance to the rank-one cone obeys
\[
 \operatorname{tail}_2(M)^2=h^2-\sigma^2
 \le2h(h-\sigma)
 \le2h\left(\frac{h^2}{\sigma}-h\right).
\]
Restrict $M$ to selected rows and columns to obtain $M^J$.
If $\operatorname{rank}R\le1$, then $\operatorname{rank}R^J\le1$ and
\[
 \|M^J-R^J\|_F\le\|M-R\|_F.
\]
Taking the infimum first over full rank-one matrices $R$ and then over
all rank-one matrices on the restricted sets gives
\[
 \operatorname{dist}_F(M^J,\{S:\operatorname{rank}S\le1\})
 \le
 \operatorname{dist}_F(M,\{R:\operatorname{rank}R\le1\}).
\]

{Let $M^{J}_{s,j}$ denote the restriction of the full}
counting-measure slice $M_{s,j}$ to the coordinates belonging to $J$.
The conditional matrix of $F=\rho^{-1}T|_J$, regarded as an operator
between probability $L^2$ spaces on sets of cardinalities
$m^{s-1}$ and $m^{d-s}$, has Hilbert--Schmidt norm
\[
 \left(\frac{1}{m^{d-1}}
 {\sum_{\substack{i\in J\\i_s=j}}|F(i)|^2}\right)^{1/2}
 =\frac{\rho^{-1}}{m^{(d-1)/2}}\|M^{J}_{s,j}\|_F
 =m\|M^{J}_{s,j}\|_F,
\]
because $\rho=m^{-(d+1)/2}$. The distance to the rank-one cone scales
by the same factor $m$. Averaging its square over $j\in J_s$ gives,
for $s\in\{2,\ldots,d-1\}$,
\begin{align*}
 \Delta_s(F)^2
 &=\frac1m\sum_{j\in J_s}
    \bigl(m\,\operatorname{tail}_2(M^{J}_{s,j})\bigr)^2\\
 &\le m\sum_{j\in J_s}\operatorname{tail}_2(M_{s,j})^2\\
 &\le2m\sum_{j\in J_s}h_{s,j}
       \left(\frac{h_{s,j}^2}{\|M_{s,j}\|_{\mathrm{op}}}-h_{s,j}\right)\\
 &\le2m\left(\max_{j\in J_s}h_{s,j}\right)
       \sum_{j\in J_s}
       \left(\frac{h_{s,j}^2}{\|M_{s,j}\|_{\mathrm{op}}}-h_{s,j}\right)\\
 &\le2m\left(\max_{j\in J_s}h_{s,j}\right)({\Gamma_s}-A_s)
 \le C_d\varepsilon.
\end{align*}
Terms with $h_{s,j}=0$ are interpreted as zero. In the penultimate
line, nonnegativity permits enlarging the sum from $J_s$ to $[n]$.
Lemma~\ref{lem:local-path-perturbation}, together with
{\eqref{eq:sqrt-cube-input-one} and $\Delta_s(F)^2\le C_d\varepsilon$, gives unimodular $H_r$ such that}
\[
 Z=\prod_{r=1}^{d-1}H_r,\qquad
 \|F-Z\|_{L^2(J)}\le C_d\sqrt\varepsilon.
\]
For $d=2$, take $Z=\operatorname{ph}(F)$ pointwise.
Define $P=\rho Z$ on $J$ and $P=0$ outside $J$. Since $|Z|=1$ and $\rho m^{d/q}=1$,
\[
 \|P\|_q=\rho m^{d/q}=1.
\]
On the box, counting and probability norms are related by
\[
 \|\rho(F-Z)\|_{\ell_q(J)}
 =\rho m^{d/q}\|F-Z\|_{L^q(J)}.
\]
Since $q<2$ and the measure on $J$ is a probability measure,
$\|F-Z\|_{L^q(J)}\le\|F-Z\|_{L^2(J)}$.  Moreover
\[
 \rho m^{d/q}=m^{-(d+1)/2}m^{(d+1)/2}=1.
\]
Outside $J$ one has $P=0$, so
\[
 \|T-P\|_q
 \le \|T-T\mathbf1_J\|_q
      +\|T\mathbf1_J-P\|_q
 \le C_d\varepsilon^{1/q}+C_d\sqrt\varepsilon
 \le C_d\sqrt\varepsilon,
\]
because $0\le\varepsilon\le1$ and $1/q>1/2$.

{For each $r$, define}
\[
 \mathsf A_r(a,b)=\mathbb E\left[
 F(i)\overline{{\prod_{\substack{1\le s\le d-1\\s\ne r}}}H_s(i_s,i_{s+1})}
 \mid i_r=a,\ i_{r+1}=b\right].
\]
Lemma~\ref{lem:robust-edge-extraction} gives
$\beta(\mathsf A_r)\le m^{3/2}$, and conditional Jensen gives
\begin{equation}\label{eq:edge-second-moment}
 \mathbb E|\mathsf A_r|^2\le\mathbb E|F|^2\le1+C_d\varepsilon,
 \qquad
 \|\mathsf A_r-H_r\|_{L^2}\le\|F-Z\|_{L^2}
 \le C_d\sqrt\varepsilon.
\end{equation}
{Since $|Z|=1$,}
\begin{align}
 \operatorname{Re}\mathbb E(\mathsf A_r\overline{H_r})
 &=\operatorname{Re}\mathbb E(F\overline Z)\\
 &=\frac{\mathbb E|F|^2+1-\mathbb E|F-Z|^2}{2}
 \ge1-C_d\varepsilon.\label{eq:edge-overlap}
\end{align}
Since $|H_r|=1$, we have $\|H_r\|_{L^4}=1$.
H\"older's inequality with exponents $4/3$ and $4$ gives
\[
 1-C_d\varepsilon
 \overset{\eqref{eq:edge-overlap}}{\le} \operatorname{Re}\mathbb E(\mathsf A_r\overline{H_r})
 \le \mathbb E|\mathsf A_rH_r|
 \le \|\mathsf A_r\|_{L^{4/3}}\|H_r\|_{L^4}
 =\|\mathsf A_r\|_{L^{4/3}}.
\]
Thus $\|\mathsf A_r\|_{L^{4/3}}\ge1-C_d\varepsilon$. Since $\mathsf A_r$ has
$m^2$ entries and $L^{4/3}$ uses uniform probability measure,
\[
 \|\mathsf A_r\|_{\ell_{4/3}}
 =\left(m^2\mathbb E|\mathsf A_r|^{4/3}\right)^{3/4}
 =m^{3/2}\|\mathsf A_r\|_{L^{4/3}}.
\]
Lemma~\ref{lem:matrix-interpolation} and $\beta(\mathsf A_r)\le m^{3/2}$ therefore give
\begin{equation}\label{eq:edge-nuclear-lower}
 \|\mathsf A_r\|_{S_1}
 \ge\frac{\|\mathsf A_r\|_{\ell_{4/3}}^2}{\beta(\mathsf A_r)}
 \ge\frac{m^3(1-C_d\varepsilon)^2}{m^{3/2}}
 \ge m^{3/2}(1-C_d\varepsilon).
\end{equation}
Thus \(\mathsf A_r\) has nearly maximal trace norm at the scale dictated by
its Frobenius energy. Its polar unitary is therefore the natural scaled
unitary against which to measure the residual error.
Let $\mathsf A_r=V_r|\mathsf A_r|$ be a polar decomposition and extend the partial isometry $V_r$ to a unitary $U_r$ on $\mathbb K^m$. Put $K_r=\sqrt m\,U_r$. Then
$\|K_r\|_F^2=m^2$ and
$\operatorname{Re}\operatorname{Tr}(K_r^*\mathsf A_r)
=\sqrt m\,\|\mathsf A_r\|_{S_1}$.  {By \eqref{eq:edge-second-moment} and \eqref{eq:edge-nuclear-lower},}
\begin{align}
 \frac{\|\mathsf A_r-K_r\|_F^2}{m^2}
 &=\frac{\|\mathsf A_r\|_F^2}{m^2}+1
   -\frac{2\|\mathsf A_r\|_{S_1}}{m^{3/2}}\\
 &=\mathbb E|\mathsf A_r|^2+1
   -\frac{2\|\mathsf A_r\|_{S_1}}{m^{3/2}}\\
 &\le C_d\varepsilon.\label{eq:polar-distance}
\end{align}
{Moreover,}
\[
 \|\mathsf A_r-H_r\|_{L^2([m]^2)}
   =\frac{\|\mathsf A_r-H_r\|_F}{m}\le C_d\sqrt\varepsilon,
\]
and \eqref{eq:polar-distance} gives
\[
 \frac{\|\mathsf A_r-K_r\|_F}{m}\le C_d\sqrt\varepsilon.
\]
The triangle inequality now gives, for $1\le r\le d-1$,
\[
 \frac{\|H_r-K_r\|_F}{m}
 \le\frac{\|H_r-\mathsf A_r\|_F}{m}
    +\frac{\|\mathsf A_r-K_r\|_F}{m}
 \le C_d\sqrt\varepsilon.
\]
{Telescoping gives}
\[
 \prod_{r=1}^{d-1}H_r-\prod_{r=1}^{d-1}K_r
 =\sum_{r=1}^{d-1}
 \left(\prod_{t=1}^{r-1}K_t\right)(H_r-K_r)
 \left(\prod_{t=r+1}^{d-1}H_t\right).
\]
Lemma~\ref{lem:edge-replacement}, applied with
$D_r=H_r-K_r$, gives for every $r$ the exact identity
\[
 \left\|
 \left(\prod_{t=1}^{r-1}K_t\right)(H_r-K_r)
 \left(\prod_{t=r+1}^{d-1}H_t\right)
 \right\|_{L^2(J)}
 =\frac{\|H_r-K_r\|_F}{m}.
\]
Since $q<2$, $\|\cdot\|_{L^q(J)}\le\|\cdot\|_{L^2(J)}$, and therefore
\[
 \|P-Q\|_q
 =\rho m^{d/q}\|Z-\textstyle\prod_{r=1}^{d-1}K_r\|_{L^q(J)}
 \le\sum_{r=1}^{d-1}\frac{\|H_r-K_r\|_F}{m}
 \le C_d\sqrt\varepsilon,
\]
where
\[
 \rho m^{d/q}=m^{-(d+1)/2}m^{d/q}
 =m^{-(d+1)/2}m^{(d+1)/2}=1.
\]
Lemma~\ref{lem:unitary-chain} gives $\|Q\|_{\mathrm{cb}}=1$.
\end{proof}

\section{Proof of Theorem A}

\begin{proof}[Proof of Theorem~\ref{thm:classification}]
By homogeneity, assume \(\|T\|_{\mathrm{cb}}=1\). Equality in
\eqref{eq:cbBH-intro} gives \(\|T\|_q=1\), while
Lemma~\ref{lem:spectral-input} gives \(A_s\le1\) for every \(s\).
Lemma~\ref{sharp-lem-cube}, applied to
\(a_i=|T(e_{i_1},\ldots,e_{i_d})|\) with \(\varepsilon=0\), therefore
provides an integer \(1\le m\le n\), sets
\(J_1,\ldots,J_d\subseteq[n]\) with
\(|J_1|=\cdots=|J_d|=m\), and
\[
 J:=J_1\times\cdots\times J_d,
 \qquad \rho=m^{-(d+1)/2},
\]
such that \(T\) vanishes outside \(J\), while for
\(F=\rho^{-1}T|_J\) one has \(|F(i_1,\ldots,i_d)|=1\) for every \(i\in J\).  {Equation \eqref{sharp-eq-full-slices} gives the full-slice bounds.}

For this cube,
\[
 \|T-T\mathbf1_J\|_q=0,\qquad
 \mathbb E_J(|F|-1)^2=0,\qquad
 \mathbb E_J|F|^2=1,
\]
and \eqref{sharp-eq-full-slices} gives
$\max_{j\in J_s}h_{s,j}\le C_d/m$ for every $s$. Thus all hypotheses of
Lemma~\ref{lem:sqrt-spectral-polar-closure} hold with $\varepsilon=0$.

{Lemma~\ref{lem:sqrt-spectral-polar-closure}, with \(\varepsilon=0\), gives unimodular matrices \(H_r\), indexed by}
\(J_r\times J_{r+1}\), and unitary matrices \(U_r\) of order \(m\), for
\(r=1,\ldots,d-1\), such that
\[
 T(e_{i_1},\ldots,e_{i_d})
 =m^{-(d+1)/2}\prod_{r=1}^{d-1}H_r(i_r,i_{r+1})
 \qquad ((i_1,\ldots,i_d)\in J)
\]
and
\[
 \|H_r-\sqrt m\,U_r\|_F=0
 \qquad (r=1,\ldots,d-1).
\]
Hence \(H_r=\sqrt m\,U_r\), so
\[
 H_rH_r^*=mI_m.
\]
Since its entries are unimodular, each \(H_r\) is Hadamard.
Restoring the normalization and absorbing the scalar into \(\lambda\)
gives \eqref{eq:classification}.

Conversely, suppose that \(T\) has the form \eqref{eq:classification}.
Put \(U_r=m^{-1/2}G_r\).  Then each \(U_r\) is unitary, and the normalized
tensor
\[
 Q(e_{i_1},\ldots,e_{i_d})=\frac1m\prod_{r=1}^{d-1}U_r(i_r,i_{r+1})
     =m^{-(d+1)/2}\prod_{r=1}^{d-1}G_r(i_r,i_{r+1})
\]
has completely bounded norm one by Lemma~\ref{lem:unitary-chain}.
Its \(m^d\) nonzero coefficients all have modulus
\(m^{-(d+1)/2}\), hence \(\|Q\|_q=1\).  Homogeneity gives
\[
 \|T\|_q=\|T\|_{\mathrm{cb}}
 =|\lambda|m^{(d+1)/2}.
\]
Thus equality holds in \eqref{eq:cbBH-intro}.
\end{proof}

\begin{corollary}
Every nonzero extremizer has exactly $m^d$ nonzero coefficients for an integer $m$ for which a Hadamard matrix of order $m$ over $\mathbb K$ exists. The case $m=1$ consists precisely of coordinate monomials. Conversely, every tensor of the form \eqref{eq:classification} is an extremizer.
\end{corollary}
\begin{proof}
By Theorem~\ref{thm:classification}, the support is
$J_1\times\cdots\times J_d$ with $|J_s|=m$ for every $s$. Therefore
\[
 |\operatorname{supp}T|=\prod_{s=1}^d|J_s|=m^d.
\]
For $m=1$, the support is a single multi-index, so $T$ is a coordinate
monomial. Conversely, sufficiency in Theorem~\ref{thm:classification}
shows that every tensor of the form \eqref{eq:classification} attains
equality.
\end{proof}

\begin{corollary}
Fix $d\ge2$ and $n\ge2$. A real full-support extremizer on $[n]^d$ exists if and only if there is a real Hadamard matrix of order $n$.
\end{corollary}
\begin{proof}
By Theorem~\ref{thm:classification}, a real extremizer has support
$J_1\times\cdots\times J_d$, where $J_s\subseteq[n]$ and $|J_s|=m$.
Full support forces $J_s=[n]$ and $m=n$, so the theorem supplies
real Hadamard matrices of order $n$.

Conversely, let $G$ be a real Hadamard matrix of order $n$, take
$J_s=[n]$ for all $s$, and set $G_r=G$ for $1\le r\le d-1$ in
\eqref{eq:classification}. The sufficiency part of
Theorem~\ref{thm:classification} then gives a real full-support
extremizer.
\end{proof}

\begin{proposition}[Gauge invariance of the representation]
Assume $d\ge3$. Let a tensor be represented as in \eqref{eq:classification}. Fix $1\le r\le d-2$ and let $D$ be a diagonal unitary matrix indexed by $J_{r+1}$. Replacing
\[
 G_r\quad\text{by}\quad G_rD,
 \qquad
 G_{r+1}\quad\text{by}\quad D^*G_{r+1}
\]
does not change any coefficient of the tensor. Simultaneously permuting the columns of $G_r$ and the rows of $G_{r+1}$ corresponds to relabelling the coordinate set $J_{r+1}$. A unimodular scalar may likewise be transferred between $\lambda$ and any one edge.
\end{proposition}
\begin{proof}
Write $D=\operatorname{diag}(\omega_j)_{j\in J_{r+1}}$, with $|\omega_j|=1$. For every active multi-index,
\begin{align*}
 (G_rD)(i_r,i_{r+1})(D^*G_{r+1})(i_{r+1},i_{r+2})
 &=G_r(i_r,i_{r+1})\omega_{i_{r+1}}
   \overline{\omega_{i_{r+1}}}G_{r+1}(i_{r+1},i_{r+2})\\
 &=G_r(i_r,i_{r+1})G_{r+1}(i_{r+1},i_{r+2}).
\end{align*}
The remaining factors are unchanged. A common permutation of the
shared index $i_{r+1}$ gives the second assertion. Finally, multiplying
one edge by a unimodular scalar and $\lambda$ by its inverse leaves
the tensor unchanged.
\end{proof}

\section{Proof of Theorem B}
\begin{proof}[Proof of Theorem~\ref{thm:flat-stability}]
Set $a_i=|T(e_{i_1},\ldots,e_{i_d})|$ for $i\in[n]^d$.
For each slot $s$ and $j\in[n]$, the marginal in
Lemma~\ref{sharp-lem-cube} is
\[
 R_s(j)=\left(\sum_{\substack{i\in[n]^d\\ i_s=j}}|T(e_{i_1},\ldots,e_{i_d})|^2\right)^{1/2}
       =\|M_{s,j}\|_F=h_{s,j}.
\]
{Thus}
\[
 A_s=\sum_{j=1}^n h_{s,j}.
\]
Lemma~\ref{lem:spectral-input} gives $A_s\le\|T\|_{\rm cb}=1$ for
every $s\in\{1,\ldots,d\}$.  Moreover
$\|a\|_{\ell_q}=\|T\|_q\ge1-\varepsilon$.

{Lemma~\ref{sharp-lem-cube} supplies sets $J_1,\ldots,J_d$ of common cardinality. Put}
\[
 J=J_1\times\cdots\times J_d,\qquad |J_s|=m,
 \qquad \rho=m^{-(d+1)/2}.
\]
Because $q=2d/(d+1)$, one has
\[
 (m^d)^{-1/q}=m^{-d/q}=m^{-(d+1)/2}=\rho,
\]
{and \eqref{sharp-eq-tail} gives}
\[
 \|T-T\mathbf1_J\|_q^q
   ={\sum_{(i_1,\ldots,i_d)\in[n]^d\setminus J}}|T(e_{i_1},\ldots,e_{i_d})|^q\le C_d\varepsilon,
\]
and therefore
\[
 \|T-T\mathbf1_J\|_q\le C_d\varepsilon^{1/q}
\]
after changing $C_d$.  If $F=\rho^{-1}T|_J$, then
{$|F(i_1,\ldots,i_d)|=|T(e_{i_1},\ldots,e_{i_d})|/\rho$. By \eqref{sharp-eq-variance} and \eqref{sharp-eq-moments},}
\[
 \mathbb E_J(|F|-1)^2\le C_d\varepsilon,
 \qquad
 |\mathbb E_J|F|^2-1|\le C_d\varepsilon.
\]
{Equation \eqref{sharp-eq-full-slices} gives}
\[
 \max_{j\in J_s}h_{s,j}\le \frac{C_d}{m}
 \qquad(s=1,\ldots,d).
\]
{Lemma~\ref{lem:sqrt-spectral-polar-closure} now gives unimodular matrices $H_r$ and a tensor $P$ satisfying}
\[
 \|T-P\|_q\le C_d\sqrt\varepsilon,
 \qquad \|P\|_q=1,
\]
{and \eqref{eq:main-tail} follows from \eqref{sharp-eq-tail}.}
\end{proof}

\section{Proof of Theorem C}
\begin{proof}[Proof of Theorem~\ref{thm:unitary-stability}]
{Use the sets $J_s$ obtained in the proof of Theorem~\ref{thm:flat-stability}. For these sets, Lemma~\ref{lem:sqrt-spectral-polar-closure} provides the unimodular matrices $H_r$ and the scaled unitaries $K_r=\sqrt m\,U_r$ simultaneously, with
\[
 \max_{1\le r\le d-1}\frac{\|H_r-\sqrt m\,U_r\|_F}{m}
 \le C_d\sqrt\varepsilon.
\]
The tensor $Q$ from that lemma has coefficients
\[
 Q(e_{i_1},\ldots,e_{i_d})
 =\rho\prod_{r=1}^{d-1}K_r(i_r,i_{r+1})
 =\frac1m\prod_{r=1}^{d-1}U_r(i_r,i_{r+1})
 \qquad(i\in J),
\]
where the second identity uses
$\rho m^{(d-1)/2}=m^{-1}$. It vanishes outside $J$ and satisfies
\[
 \|T-Q\|_q\le C_d\sqrt\varepsilon,\qquad \|Q\|_{\mathrm{cb}}=1.
\]
Lemma~\ref{lem:unitary-chain} gives $\|Q\|_F=m^{-1/2}$.
Over $\mathbb R$, phase choices give signs and polar factors can be
chosen orthogonal.}
\end{proof}

\section{Sharpness of Theorems B and C}
\label{sharp-sec-optimality}

\begin{proposition}
\label{sharp-prop-optimality}
{Fix $d\ge2$ and put $q=2d/(d+1)$. There are $d$-linear forms
$T_t:(\mathbb R^2)^d\to\mathbb R$, $0<t<1/2$, such that}
\[
 \|T_t\|_{\mathrm{cb}}=1,\qquad
1-\|T_t\|_q=\frac{t^2}{d+1}+O_d(t^4),
\]
whose coefficient $\ell_q$ distance from every normalized unitary chain
and every normalized flat unimodular path supported on an equal-sided
Cartesian box in $[2]^d$ is at least $c_dt$, where $c_d>0$ depends only on $d$, for sufficiently small $t$.
The same lower bound holds when the approximants are allowed to be complex,
after viewing $T_t$ as a complex multilinear form.
Consequently no uniform estimate for either structural approximation
can have deficit exponent greater than $1/2$.
\end{proposition}

\begin{proof}
Let
\[
 F_2=\begin{pmatrix}1&1\\1&-1\end{pmatrix},\qquad
 \sigma_1=1,\quad\sigma_2=-1,
\]
and define on $\{1,2\}^d$
\begin{equation}
 T_t(e_{i_1},\ldots,e_{i_d})
 =2^{-(d+1)/2}\frac{1+\sigma_{i_1}t}{\sqrt{1+t^2}}
       \prod_{r=1}^{d-1}F_2(i_r,i_{r+1}).
 \label{sharp-eq-all-degree-example}
\end{equation}
For arbitrary contractions
$X^s_1,X^s_2$ on a common Hilbert space, put
\[
 L_t=\frac{[(1+t)X^1_1\ \ (1-t)X^1_2]}{\sqrt{2+2t^2}},
 \qquad
 R=\frac{[X^d_1\ \ X^d_2]^{\mathsf T}}{\sqrt2},
 \qquad
 D_s=\operatorname{diag}(X^s_1,X^s_2)\quad(2\le s\le d-1).
\]
The row $L_t$, column $R$, and block diagonal matrices $D_s$ are
contractions, and $V=F_2/\sqrt2$ acts unitarily on the direct sum.
The cb test for \eqref{sharp-eq-all-degree-example} is
\[
 L_t V D_2 V\cdots D_{d-1} V R,
\]
with the internal diagonal factors omitted when $d=2$.  Hence
$\|T_t\|_{\mathrm{cb}}\le1$.  On the other hand, each of the two last-slot slices has Frobenius norm $1/2$: directly,
\[
 {\sum_{i_1,\ldots,i_{d-1}=1}^2}
 |T_t(e_{i_1},\ldots,e_{i_{d-1}},e_j)|^2
 =\frac{2^{d-2}2^{-(d+1)}}{1+t^2}
       \bigl((1+t)^2+(1-t)^2\bigr)
 =\frac14.
\]
The mixed slice bound $A_d(T_t)\le\|T_t\|_{\mathrm{cb}}$ gives the
reverse inequality. Thus $\|T_t\|_{\mathrm{cb}}=1$.

Its coefficient norm is
\[
 \|T_t\|_q
 =\frac{\left(\dfrac{(1+t)^q+(1-t)^q}{2}\right)^{1/q}}
        {\sqrt{1+t^2}}
 =1-\frac{2-q}{2}t^2+O_d(t^4)
 =1-\frac{t^2}{d+1}+O_d(t^4).
\]
The first-slot slice Frobenius norms of $T_t$ are
\[
 \frac{1+t}{2\sqrt{1+t^2}},\qquad
 \frac{1-t}{2\sqrt{1+t^2}}.
\]
For every normalized unitary chain supported on the full box $[2]^d$
\[
 Q(e_{i_1},\ldots,e_{i_d})=\frac12\prod_{r=1}^{d-1}U_r(i_r,i_{r+1}),
\]
and every fixed $i_1\in\{1,2\}$,
\begin{align*}
 \sum_{i_2,\ldots,i_d=1}^2|Q(e_{i_1},\ldots,e_{i_d})|^2
 &=\frac14\sum_{i_2=1}^2|U_1(i_1,i_2)|^2
   \sum_{i_3=1}^2|U_2(i_2,i_3)|^2\cdots
   \sum_{i_d=1}^2|U_{d-1}(i_{d-1},i_d)|^2\\
 &=\frac14.
\end{align*}
Thus each first-slot slice has Frobenius norm $1/2$. The reverse
triangle inequality within the two slices therefore gives
\[
 \|T_t-Q\|_F^2\ge
 \left(\frac{1+t}{2\sqrt{1+t^2}}-\frac12\right)^2+
 \left(\frac{1-t}{2\sqrt{1+t^2}}-\frac12\right)^2=1-\frac1{\sqrt{1+t^2}}\ge c_dt^2.
\]
Since $q<2$, this also bounds the coefficient $\ell_q$ distance from
below. Every normalized flat unimodular path on $[2]^d$ also has first-slot slice Frobenius norms $1/2$, so the same argument applies to both comparison classes.

Since every coordinate set is $[2]$, the only smaller equal-sided Cartesian boxes have side $m=1$ and hence are singletons. Their normalized approximants have Frobenius norm one,
whereas $\|T_t\|_F=1/\sqrt2$, so their distance from $T_t$ is bounded
away from zero.
\end{proof}

{
\section*{Index of notation}
\addcontentsline{toc}{section}{Index of notation}
\begingroup
\small
\renewcommand{\arraystretch}{1.22}
\noindent\begin{tabularx}{\textwidth}{@{}>{\raggedright\arraybackslash}p{29mm}>{\raggedright\arraybackslash}X>{\raggedright\arraybackslash}p{28mm}@{}}
\toprule
Symbol & Meaning & Defined in\\
\midrule
$\|T\|$ & Classical multilinear supremum norm on $(\ell_\infty^n)^d$ & Introduction\\
$\|T\|_{\mathrm{cb}}$ & Supremum over products of contractions in the fixed slot order & Definition~\ref{def:cb-norm-intro}\\
$q=q_d$, $\|T\|_q$ & $q_d=2d/(d+1)$ and the coefficient $\ell_q$ norm & \eqref{eq:coefficient-q-intro}\\
$\ell_p(E)$, $L^p(E)$ & Counting norm and uniform probability norm & \eqref{eq:counting-probability-norms}\\
$T|_J$, $T\mathbf1_J$ & Restriction to $J$ and its extension by zero & Section~\ref{sec:spectral-input}\\
$M_{s,j}$, $M^J_{s,j}$ & Full counting-measure slice and its restriction to $J$ & \eqref{eq:slice-matrix-definition}\\
$\mathsf M_{s,j}(F)$ & Conditional kernel acting between probability $L^2$ spaces & \eqref{eq:probability-kernel-convention}\\
$\operatorname{tail}_2(M)$ & Frobenius distance to matrices of rank at most one & \eqref{eq:rank-one-tail-definition}\\
$\|M\|_{S_1}$ & Trace norm, the sum of the singular values & \eqref{eq:trace-norm-definition}\\
$h_{s,j}$, $A_s$, $\Gamma_s$ & Full-slice Frobenius norm, mixed-norm sum, and spectral sum & \eqref{eq:slice-quantities}\\
$m$, $N$, $\rho$ & Common side length, $N=m^d$, and $\rho=m^{-(d+1)/2}$ & Section~\ref{subsec:flat-scale-bookkeeping}\\
$\Delta_s(F)$ & Root-mean-square conditional rank-one error & \eqref{eq:conditional-rank-defect}\\
$\beta(A)$ & Bilinear supremum over two families of unit-ball vectors & \eqref{eq:beta-definition}\\
$\mathsf A_r$ & Averaged matrix obtained after cancelling all other path edges & \eqref{eq:edge-average-definition}\\
$H_r$, $U_r$, $K_r$ & Unimodular edges, unitaries, and $K_r=\sqrt m\,U_r$ & Lemma~\ref{lem:sqrt-spectral-polar-closure}\\
$\mathbb T$, $\operatorname{Diag}(\xi)$ & Unit circle and the diagonal matrix with diagonal $\xi$ & Remark~\ref{rem:classical-norm-complex-hadamard-chains}\\
\bottomrule
\end{tabularx}\par
\endgroup
}

\section*{Author contributions}
All authors contributed to the conception of the problem, the mathematical development, the verification of the arguments, and the preparation and revision of the manuscript. All authors read and approved the final version.

\section*{Funding}
D. N\'u\~nez-Alarc\'on and D. M. Pellegrino were partially supported by CNPq Grants 406457/2023-9, 403964/2024-5. In addition, D. Pellegrino was partially supported by CNPq Grant 305807/2025-0. E. V. Teixeira acknowledges support from the Grayce B. Kerr Chair at Oklahoma State University and partial support from the DARPA ExpMath program under Agreement No. HR0011262E029.

\section*{Competing interests}
The authors declare that they have no competing interests.

\section*{Data availability}
Data sharing is not applicable to this article, as no datasets were generated or analyzed during the current study.

\begin{samepage}
\section*{Declaration on the use of generative AI}
During the preparation of this manuscript, the authors used a generative-AI tool for exploratory calculations, consistency checks, organization of arguments, and drafting and editorial assistance. All mathematical statements, proofs, references, and final formulations were reviewed and verified by the authors, who take full responsibility for the content of the manuscript.

\end{samepage}

\end{document}